\documentclass[11pt]{amsart}
\usepackage[english]{babel}
\usepackage{amsmath}
\usepackage{amsfonts}
\usepackage{amssymb}
\usepackage[sc]{mathpazo}
\usepackage{eulervm}
\usepackage{amsthm}
\usepackage{color}
\usepackage{xargs}

\usepackage{float}

\usepackage[english]{babel}
\usepackage{yfonts}
\usepackage[T1]{fontenc}
\usepackage[utf8x]{inputenc}
\usepackage{enumerate}
\usepackage{verbatim}
\usepackage{graphicx}
\usepackage{verbatim}
\usepackage{faktor}
\usepackage{xcolor}
\usepackage{xfrac}
\usepackage{tikz}
\usepackage[all]{xy}
\usepackage[mathscr]{euscript}

\usepackage{graphicx}
\usepackage{epstopdf}
\usepackage{psfrag}
\newcommand{\compl}[1]{}

\newcommand{\eps}{\varepsilon}

 \newcommand{\R}{\mathbb R}

\DeclareMathOperator{\PSL}{PSL}
\DeclareMathOperator{\SL}{SL}

\DeclareMathOperator{\PGL}{PGL}
\DeclareMathOperator{\PSU}{PSU}
\DeclareMathOperator{\Sp}{Sp}

\renewcommand{\H}{\mathbb H}   % Hyperbolic plane H
\newcommand{\bH}{\partial \H} % boundary of H

\newcommand{\Sf}{S} % Surface (no boundary, infinite volume) S
\newcommand{\calG}{{\mathcal G}} % geods in ..
\newcommand{\calDG}{{\mathcal DG}} % transverse pairs of geods

\newcommand{\calC}{{\mathcal C}} % space of currents
\newcommand{\calL}{{\mathcal L}} % Liouville current

\DeclareMathOperator{\supp}{supp} % support of a current
\newcommand{\ML}{\mathcal{ML}} % Space of measured laminations
\newcommand{\MLc}{\ML_c} % with compact support

\newcommand{\Lam}{\Lambda} % a lamination
\newcommand{\Lamt}{{\widetilde\Lam}} % lift of \Lam
\newcommand{\Lamref}{\Lam^{\mathrm{ref}}} 

\newcommand{\calR}{{\mathcal R}} % a complementary region
\newcommand{\calRt}{\wt{\calR}} % lift
\newcommand{\calQ}{{\mathcal Q}} % another complementary region
\newcommand{\calQt}{\wt{\calQ}} % lift

\newcommand{\calE}{\mathcal E} % set of special geodesics
\DeclareMathOperator{\Syst}{Syst} % Systole

\DeclareMathOperator{\Graph}{Graph} % intersection graph

\DeclareMathOperator{\Hull}{Hull} % convex hull

\newcommand{\D}{\mathbb D}
\newcommand{\G}{\Gamma}
\newcommand{\C}{\mathbb C}
\newcommand{\F}{\mathbb F}

\DeclareMathOperator{\Carr}{Carr}
\newcommand{\calX}{\mathcal X}%character variety
\newcommand{\calV}{\mathcal V}
\newcommand{\LL}{\mathbb L}
\newcommand{\ov}{\overline}
 \newcommand{\wt}{\widetilde}

\newcommand{\calT}{\mathcal T}
\newcommand{\calF}{\mathcal F}
 \newcommand{\K}{\mathbb K}
 \newcommand{\Id}{{\rm Id}}
 \newcommand{\Z}{\mathbb Z}
 \renewcommand{\P}{\mathbb P}
 \newcommand{\N}{\mathbb N}
 \newcommand{\Dd}{\mathcal D}
\newcommand{\calB}{\mathcal B}
\newcommand{\Hom}{\mathrm {Hom}}
 \newcommand{\mcg}{\mathcal{MCG}}

\newcommand{\bq}{\begin{equation}}
\newcommand{\eq}{\end{equation}}
\newcommand{\bqn}{\begin{equation*}}
\newcommand{\eqn}{\end{equation*}}
\newcommand{\ba}{\begin{aligned}}
\newcommand{\ea}{\end{aligned}}
\newcommand{\be}{\begin{enumerate}}
\newcommand{\ee}{\end{enumerate}}
\newcommand{\bei}{\begin{itemize}}
\newcommand{\eei}{\end{itemize}}
\newcommand{\bsm}{\left(\begin{smallmatrix}}
\newcommand{\esm}{\end{smallmatrix}\right)}                   
\newcommand{\bpm}{\begin{pmatrix}}
\newcommand{\epm}{\end{pmatrix}}

\newcommand{\tr}{\operatorname{tr}}

\theoremstyle{plain}
\newtheorem{Theorem}{Theorem}[section]
\newtheorem{Lemma}[Theorem]{Lemma}
\newtheorem{Proposition}[Theorem]{Proposition}
\newtheorem{Corollary}[Theorem]{Corollary}
\newtheorem*{ClosingLemma}{Closing Lemma}
\newtheorem*{claim}{Claim}

\theoremstyle{definition}
\newtheorem{Example}[Theorem]{Example}
\newtheorem{Examples}[Theorem]{Examples}

\newtheorem{Remark}[Theorem]{Remark}

\numberwithin{equation}{section}

\newcommand{\thismonth}{\ifcase\month % case 0 --- impossible!
  \or January\or February\or March\or April\or May\or June%
  \or July\or August\or September\or October\or November%
  \or December\fi}

\date{\today}
\begin{document}

%\title[A structure theorem for geodesic currents] % short running title
%      {A structure theorem for geodesic currents and length spectrum
%        compactifications} 
%%% Full title
  % A decomposition theorem for geodesic currents
  
\title[Currents, Systoles, and Compactifications] % short running title
      {Currents, Systoles, \\ and Compactifications of Character Varieties} 
  
\author[]{M. Burger}
\address{Department Mathematik, ETH Zentrum, 
R\"amistrasse 101, CH-8092 Z\"urich, Switzerland}
\email{burger@math.ethz.ch}

\author[]{A. Iozzi}
\address{Department Mathematik, ETH Zentrum, 
R\"amistrasse 101, CH-8092 Z\"urich, Switzerland}
\email{iozzi@math.ethz.ch}

\author[]{A. Parreau}
\address{Univ. Grenoble Alpes, CNRS, Institut Fourier, F-38000
  Grenoble, France}
\email{Anne.Parreau@univ-grenoble-alpes.fr}

\author[]{M. B. Pozzetti}
\address{Mathematical Institute, Heidelberg University, Im Neuenheimer feld 205, 69120 Heidelberg, Germany }
\email{pozzetti@mathi.uni-heidelberg.de}
 %MSC2020 57K20 (Teichmueller, curve complexes etc.  22E40  Discrete subgroups of Lie groups  )

\thanks { Marc Burger thanks Francis Bonahon and Kasra Rafi for enlightening conversations on geodesic currents and Kasra Rafi for suggesting to prove a decomposition theorem for geodesic currents. Beatrice Pozzetti thanks Anna Wienhard and Darryl Cooper for insightful conversations.
\\
\indent
Beatrice Pozzetti was partially supported by SNF grant P2EZP2\_159117, 
and by DFG project PO 2181/1. Marc Burger and Alessandra Iozzi were partially supported by SNF grant 2-77196-16. 
Alessandra Iozzi acknowledges moreover support from U.S. National Science Foundation grants DMS 1107452, 1107263, 1107367 
"RNMS: Geometric Structures and Representation Varieties" (the GEAR Network). 
Anne Parreau thanks the Forschungsinstitut f\"ur Mathematik for their hospitality.
Marc Burger thanks the Leverhulme Trust for supporting his visit to the University of Cambridge as Leverhulme Visiting Professor.  
Marc Burger, Alessandra Iozzi and Beatrice Pozzetti thank the Isaac Newton Institute for Mathematical Sciences, Cambridge, 
for support and hospitality during the program ``Non-Positive Curvature Group Actions and Cohomology'' where work on this paper was undertaken. 
This work was partially supported by EPSRC grant no P/K032208/1.}

\date{\today}

\maketitle

\begin{abstract}
We study the Weyl chamber length compactification both
of the Hitchin and of the maximal character varieties 
and determine therein an open set of discontinuity 
for the action of the mapping class group.
This result is obtained as consequence of a canonical decomposition of a geodesic current 
on a surface of finite type arising from a topological decomposition of the surface
along special geodesics. We show that each component either is
associated to a measured lamination or has positive systole.
For a current with positive systole, we show that  the intersection
function on the set of closed curves is bilipschitz equivalent to the
length function with respect to a hyperbolic metric.
\end{abstract}

\tableofcontents

\section{Introduction}
\label{sec:intro}

Let $S=\Gamma\backslash\H$ be a geometrically finite surface, 
where $\H$ is the hyperbolic plane and $\Gamma<\PSL(2,\R)$
is a finitely generated torsion-free discrete group.
A geodesic current on  $S$ is a $\Gamma$-invariant Radon measure 
on the space $\calG(\H)$ of unoriented, unparametrized geodesics. 
Geodesic currents occur in many different contexts. 
For instance they play a fundamental role in the study of hyperbolic structures \cite{Bon88-lam}, 
of negatively curved metrics \cite{Otal}, or of singular flat structures
\cite{DLR}.
A crucial fact is that length functions of these structures are intersection functions of
geodesic currents. This has been recently extended to Hitchin and
maximal representations by Martone and Zhang \cite{Martone_Zhang} when
$S$ is compact (see also \cite{BCLS}). 

Let $\Sigma\subset S$ be the convex core of $S$.  
Our main object of study are geodesic currents on $S$ whose
support is contained in the subset $\wt{\calG}(\Sigma)$ of geodesics
whose projection is in $\Sigma$. 
We refer to these as {\em geodesic currents on $\Sigma$} and we denote them by $\calC(\Sigma)$.
The aim of this paper is to establish structural properties of
geodesic currents  on $\Sigma$ in terms of their intersection 
with closed geodesics, and in particular in terms of their systole. 
Our motivation comes from the study of compactifications of maximal
and Hitchin character varieties:  as an application, in the case in which $S$ is compact, 
we construct a natural open domain of discontinuity for the action of the
mapping class group of $S$ on the Weyl chamber length boundary of these components and show that
in the higher rank case this set is not empty (see \S~\ref{s:Intro_TP-boundary}).  
We will also give explicit examples of actions of $\Gamma$ on $\widetilde A_2$-buildings
whose orbit maps are quasi-isometric embeddings and whose length functions
are in this set of discontinuity (see \S~\ref{subsec:ActB}).

The degree of generality adopted here in our treatment of currents, in
particular allowing $S$ to have cusps, turns out to be needed in
order to understand all the possible degenerations of maximal or
Hitchin representations in higher rank groups. 
If $\F$ is a non-Archimedean real closed field,  
the study of maximally framed representations of surface groups into $\Sp(2m,\F)$ was initiated in \cite{BP}.
In a forthcoming paper we will show how to associate
to such a representation $\rho$ a geodesic current $\mu_\rho$ whose intersection function
on closed geodesics gives the length function, \cite{BIPP2}. 
Together with the results of this current paper, 
the assignment $\rho\mapsto\mu_\rho$ is a key tool in the study of the ``real spectrum'' compactification
of maximal and Hitchin character varieties and of its nice algebraic geometric properties (see \cite{BIPP-ann} for an announcement of the results).

\subsection{Decomposition of currents}
Given a geodesic current $\mu$ on $\Sigma$, 
we exhibit two laminations with corresponding decompositions of the current
and show that the complementary regions are filled by the support of $\mu$ in a specific manner. 

\medskip
We say that a geodesic $g$ in $\Sigma$ is \emph{$\mu$-short} if no lift thereof intersects transversally a geodesic in the support of $\mu$. 
The terminology reflects the fact that, in an appropriate sense, 
the topological intersection of two geodesics generalizes to a concept of length.  
Note that geodesic currents associated to points in the Hitchin and maximal character varieties 
of a compact surface $S$ are binding, that is, they have no $\mu$-short geodesics.  
This is no longer the case for currents associated to points in the boundary of such character varieties
and the study of $\mu$-short geodesics will enable us to analyze the structure of such currents.

Given a geodesic lamination $\Lam\subset\Sigma$ consisting of $\mu$-short geodesics, the current decomposes as a finite sum 
\bq\label{eq:decomp}
\mu=\mu_\Lam+\sum_\calR \mu_\calR
\eq
that is orthogonal for the Bonahon-intersection form $i(\cdot,\cdot)$
%
%\medskip
%\noindent
%\begin{minipage}{.5\textwidth}
(see \S2 for the definition),  
where the sum is taken over the complementary regions of $\Lam$, and $\mu_\calR$ (respectively $\mu_\Lam$)
denote  the currents on $\Sigma$ given by the restriction of $\mu$ to the set of geodesics projecting into $\calR$ 
(respectively $\Lam$).
\begin{center}
\begin{figure}[h]
\begin{tikzpicture}[scale=.6]
\draw (0,0) circle [radius=3];
% Lambda
\draw (.5,2.96) node[above] {$\wt\Lambda$};
% mu_R
\draw (-2.3,-2) node[left] {$\mu_{{\mathcal R}}$};
% R
\draw (-2.5,-.9) node {$\wt{\mathcal R}$};
% mu_R'
\draw (-1.2,-2.8) node[left] {$\mu_{{\mathcal R}'}$};
% R'
\draw (1.5,0) node {$\wt{\mathcal R}'$};
% center upper left
\draw (-1,2.83) arc (380: 236: 1);
\draw[very thick, green] (-.6,2.94) arc (369:281.5:3);
\draw[very thick, green] (-.4,2.95) arc (372:277.5:3);
\draw (.5,2.96) arc (350: 313.2:9);
\draw[very thick, green] (1,2.82) arc (350: 311.7:9);
\draw[very thick, green] (1.2,2.75) arc (350: 311.3:9);
\draw[very thick, green] (1.4,2.65) arc (350: 311.2:9);
%center lower left
\draw[very thick, green] (-2.87,.87) arc (427:307:1.4);
\draw[very thick, green] (-2.82,1) arc (430:304:1.5);
\draw[very thick, green] (-2.77,1.15) arc (433:301:1.6);
\draw[very thick, green] (-2.7,1.3) arc (436:297.5:1.7);
% center lower right
\draw[very thick, green] (-1.7,-2.5) arc (143:82.5:5);
\draw[very thick, green] (-1.5,-2.6) arc (140:83.5:5);
\draw[very thick, green] (-1.3,-2.7) arc (137:85:5);
\draw (0,-3) arc (180:53:1.5);
%\center upper right
%\draw (1.66, 2.5) arc (150:258: 2);
%\draw (1.6,2.55) arc (154:251: 2.3);
\draw (1.76, 2.43) arc (154:253: 2);
\draw (1.86, 2.35) arc (158:251: 1.9);
% circle again
\draw (0,0) circle [radius=3];
\end{tikzpicture}

\caption{The lift $\wt\Lambda$ to $\H$ of the lamination $\Lambda$, the lifts $\wt{\mathcal R}$ and $\wt{\mathcal R'}$ of the regions $\mathcal R$ and $\mathcal R'$ and some geodesics in the support of $\mu_{\mathcal R}$ and $\mu_{\mathcal R'}$}
\end{figure}
\end{center}

\medskip
%If $A$ is a subset of the set $\calG(\Sigma)$ of unparametrized unoriented geodesics in $\Sigma$ with a certain property $\tau$, we call {\em $(\tau,A)$-solitary} the geodesics in $A$ that do not intersect any other geodesic in $A$.
%We hence consider the set $\Lambda_\mu$ of $(\mu\text{-short},\calG(\Sigma))$-solitary geodesics
We consider the set $\Lambda_\mu$ of \emph{solitary $\mu$-short geodesics}, namely $\mu$-short geodesics that don't intersect any other $\mu$-short geodesic
\bqn
\Lambda_\mu:=\{c\in\calG(\Sigma):\, \mu\text{-short, and }i(c,c')=0\,\,\forall c' \;\mu\text{-short}\}
\eqn
and the set $\calE_\mu$ of \emph{closed $\mu$-short solitary geodesics}:
\bqn
\calE_\mu:=\{c\in\calG(\Sigma):\,\text{ closed, }\mu\text{-short, and }i(c,c')=0\,\,\forall c'\text{ closed and }\mu\text{-short}\}\,.
\eqn
%Let $\calE_\mu$ be the set of closed $\mu$-short geodesics in $\Sigma$ 
%that are not intersected transversally by any closed $\mu$-short  geodesic:
%\bqn
%\ba
%\calE_\mu:=\{c\in\calG(\Sigma):\,c\text{ closed, }\mu\text{-short, and }i(c,c')=0\,\,\forall c'\text{ closed and }\mu\text{-short}\}
%\ea
%\eqn
%and let $\Lam_\mu$ be the set of $\mu$-short geodesics in $\Sigma$ that are not intersected transversely by any $\mu$-short geodesic. 
Then $\calE_\mu$ is a finite collection of pairwise disjoint
simple closed geodesics containing the boundary components of $\Sigma$.
In particular $\calE_\mu$ is a geodesic lamination, and $\Lamref_\mu=\calE_\mu \cup\Lam_\mu$ is a geodesic lamination refining $\calE_\mu$.   

\begin{figure}[h]
\begin{center}
\begin{tikzpicture}[scale=.6]
\draw (0,0) circle [radius=3];
% center upper right
\draw (1.6,2.55) arc (154:251: 2.3);
% center upper left
\draw (-2,2.24) arc (400:290:2);
\draw (-1,2.83) arc (380: 236: 1);
\draw (-.6,2.94) arc (369:281.5:3);
\draw (-.4,2.95) arc (372:277.5:3);
% center lower left
\draw (-2.87,.87) arc (427:307:1.4);
\draw (-2.82,1) arc (430:304:1.5);
\draw (-2.77,1.15) arc (433:301:1.6);
\draw (-2.7,1.3) arc (436:297.5:1.7);
% center lower right
\draw (2.7,1.3) arc (130:235:1.5);
\draw (2.5,1.66) arc (130:242:1.7);
\draw (2.3,1.9) arc (134.5:237:2.15);
\draw (2,2.24) arc (140:250:2);
%trying to shade R
\filldraw[fill=orange!20!white, draw=orange!50!black] (1.6,2.55) arc (151: 238: 3) arc (331: 219:3) arc (298: 379:4) arc (98:58:3);
%geodesic right
\draw[very thick, red] (1.6,2.55) arc (151: 238: 3);
\draw (1.2,0.2) node {$\sigma_2$};
%geodesic left
\draw[very thick, red] (-.4,2.95) arc (378.8:298:4); 
\draw (-1.2,-1.5) node {$\sigma_1$};
% red crossing geodesics
%\draw[very thick, red]  (0,3) -- (1,-2.82);
\draw[very thick, red] (0,3) -- (0,-3);
\draw[very thick, red] (0,3) arc (180:237.1:5.5);
%\draw (1.3,-2) node {$\sigma_3$};
\draw[very thick, red] (1.3,2.7) -- (-1,-2.83);
%\draw (-.2,-2) node {$\sigma_2$};
\draw[very thick, red] (1,2.82) arc (160: 213.3:6);
% Lambda_mu
\draw[red] (-2.6,-2) node {$\Lambda_\mu$};
\draw[red] (3,-1.6) node {$\Lambda_\mu$};
%R
\draw[orange] (1.5,-3) node {$\mathcal R$};
%mu-short
\draw (-1.5,-3.2) node {$\mu$-short};
% circle again
\draw(0,0) circle [radius=3];
\end{tikzpicture}
\end{center}
\caption{All the thick red geodesics in  $\wt{\mathcal R}$ project to $\mu$-short;  only $\sigma_1$ and $\sigma_2$ are in $\wt\Lambda_\mu$.}
\end{figure}

\begin{center}
\begin{figure}[h]
\begin{tikzpicture}[scale=.5]
%bottom side surface
\draw (-6,4) arc (235: 305: 10);
%top side surface
\draw (-6,-4) arc (125: 55: 10);
%left hand side surface
\draw (-6, 3.8) .. controls (-3,2)  and (-3,-.5)..   (-6,-3.8);
%lamination lambda
\draw (5.5, 3.8) .. controls (0.5,1.5)  and (1,-1.5) ..   (5.5,-3.65);
\draw (1.7,0) node {$\lambda$};
%right hand side surface
\draw (5.5, 3.6) .. controls (2,2)  and (1.5,-1.5) ..   (5.5,-3.6);
%right side e_mu
\draw (0,2.2)  .. controls (1,1) and (1,-1).. (0,-2.2);
%left side e_mu
\draw[dashed] (0,2.2)  .. controls (-.8,1) and (-.8,-1).. (0,-2.2);
%e_mu
\draw (0,-2.2) node[below] {$\calE_\mu$};
%bottom genus
\draw (-3,-.5) arc (240: 300: 1.5);
%top genus
\draw (-2.8,-.5) arc (120: 60:1);
%geodesic c
\draw (-6,3.9) .. controls (0,0) .. (5.5,-3.95);
\draw (-2.3,1) node {$c$};
\end{tikzpicture}
\caption{Here $\Lambda=\{\lambda\}$ and $\mu=\delta_\lambda$.  
In this case there is a simple $\mu$-short geodesic $c$ that crosses $\calE_\mu$ but it is of course not closed.\label{fig:1}}
\end{figure}
\end{center}

\begin{Example}[See Figure~\ref{fig:1}]
  \label{ex:1}
%\be
%\item 
Let $S=\Sigma$ be a hyperbolic surface of finite area with at least one cusp, 
and let $\Lam\subset S$ be a finite lamination consisting of geodesics with all their endpoints in cusps. 
Let $\mu=\sum_{c\in \Lam}\delta_c$ be the geodesic current
that corresponds to the measured lamination with a Dirac mass
$\delta_c$ along every leaf $c$ of $\Lambda$. 
% (see Example~\ref{ex: basic currents}(2) for the definition of )
In this case $\calE_\mu$ is the set of boundary components of the smallest subsurface 
with geodesic boundary containing $\Lambda$ and $\Lam_\mu=\Lambda$.

%
%\item For a Patterson-Sullivan current $\mu$ on $\Sigma$ (as in Example~\ref{ex: basic currents} (3)) 
%the set $\calE_\mu$ consists of the boundary components of $\Sigma$ and $\Lambda_\mu$ is empty.

%\ee
\end{Example}

\begin{Theorem}\label{thm:1} 
Let  $\mu$ be a geodesic current on the convex core $\Sigma$ of a geometrically finite hyperbolic surface. 
\begin{enumerate}
\item Let $\Lambda$ be the lamination  $\Lambda=\calE_\mu$ 
and let us consider the corresponding decomposition \eqref{eq:decomp} of $\mu$.
For every complementary region $\calR$ of $\calE_\mu$, either $\mu_\calR=0$
or every lift of every closed geodesic $c$ in $\calR$ intersects transversally the support $\supp(\mu)$ of $\mu$.
%    For every complementary region $\calR$ in $\Sigma\smallsetminus \calE_\mu$ either $\mu_\calR=0$ or 
%no closed geodesic $c$ in $\calR$ is $\mu$-short, equivalently  $i(c,\mu)>0$ for every such $c$. 
\item Let $\Lambda$ be the lamination $\Lambda=\Lamref_\mu$ 
and let us consider the corresponding decomposition \eqref{eq:decomp} of $\mu$.
For every complementary region $\calR$ of $\Lamref_\mu$, either $\mu_\calR=0$
or every lift of every geodesic $c$ in $\calR$ intersects transversally $\supp(\mu)$.
%For every complementary region $\calR$ of $\Lamref_\mu$ either $\mu_\calR=0$ or no geodesic in $\calR$ is $\mu$-short.
\end{enumerate}
\end{Theorem}

If $S=\Sigma$ is compact, part (1) of Theorem~\ref{thm:1} has been
established independently by Erlandsson and Mondello \cite[\S~1.5]{EM18}.

%\begin{minipage}{.5\textwidth}
\begin{center}
\begin{figure}[h]
\begin{tikzpicture}%[scale=.6]
%\draw (-3,0) -- (3,0);
%\draw (0,-3) -- (0,3);
\draw (0,0) circle [radius=2];
\draw (-.7,-1.88) arc (315: 348: 7);
\draw (-.8,-1.84) arc (315: 348: 7);
\draw (-.9,-1.8) arc (315: 348.3: 7);
\draw (-1,-1.74) arc (315: 348.25: 7);
\draw (-1.1,-1.68) arc (315: 348.2: 7);
\draw (-1.2,-1.61) arc (315: 348: 7);
\draw (-1.3,-1.54) arc (315: 347.6: 7);
\draw[rotate=70] (-.7,-1.88) arc (315: 348: 7);
\draw[rotate=70] (-.8,-1.84) arc (315: 348: 7);
\draw[rotate=70] (-.9,-1.8) arc (315: 348.3: 7);
\draw[rotate=70] (-1,-1.74) arc (315: 348.25: 7);
\draw[rotate=70] (-1.1,-1.68) arc (315: 348.2: 7);
\draw[rotate=70] (-1.2,-1.61) arc (315: 348: 7);
\draw[rotate=70] (-1.3,-1.54) arc (315: 347.6: 7);
\draw[red] (-1.68,-1.1) arc (300: 431: 1);
\draw[red] (-1.77,-.93) arc (300:431: .8);
\draw[red] (-1.86,-.78) arc (300:431: .6);
\draw[red] (-.3,-1.96) arc (171: 20: .5);
\draw[red] (1.414,1.414) arc (150: 225: 2);
\draw[red] (1.5, 1.3) arc (150: 255:.8);
\draw[red] (2,0) arc (90: 240:.5);
\draw (1.985,-.1) arc (90:260:.15);
\draw (1.975,-.2) arc (90:255:.2);
\draw (1.95,-.5) arc (85:238:.15);
\draw (-1.414,1.414) node[left] {$\mu$};
\draw[red] (-2,0) node[left] {$\wt\Lambda$};
\end{tikzpicture}
%\caption{The case $\Lambda=\{\lambda\}$ and $\mu=\delta_\lambda$.}
\caption{The lift of the lamination $\Lambda$ in Theorem \ref{thm:1} (2), and some geodesics in the support of $\mu$ (in black).}
\end{figure}
\end{center}
%\end{minipage}

\subsection{Systole of a current}
In the first decomposition in Theorem~\ref{thm:1} let $\Sigma=\Gamma\backslash C$,
where $C\subset\H$ is the convex hull of the limit set of $\Gamma$ and let $\calR$ be a complementary region of $\calE_\mu$; 
its metric completion $\Sigma'$ is then the quotient $\G'\backslash C'$ of a closed convex subset $C'\subset C$ 
by a finitely generated subgroup $\Gamma'<\G$. 
We can then restrict $\mu$ to the set of geodesics contained in $C'$ and 
obtain a current $\mu'$ on $\Sigma'$ such that $i(\mu,c)=i(\mu',c)$ for every closed geodesic $c$ contained in $\Sigma'$. 
Thus, because of Theorem~\ref{thm:1}(1), in order to study currents  on a surface with boundary 
it is enough to consider currents $\mu$ for which $i(\mu,c)>0$ for every closed geodesic $c$ contained in the interior of the surface.

\medskip
For a current $\mu$ on $\Sigma$, define then its \emph{systole}
\bqn
\Syst(\mu):=\inf i(\mu,c)
\eqn
where the infimum is taken over all closed geodesics $c$ contained in the interior $\mathring \Sigma$ of $\Sigma$.
Our next results analyze the structure of $\mu$ depending on the vanishing or the non-vanishing of $\Syst(\mu)$. 
These results are stated in terms of the intersection properties of $\mu$ with geodesic currents with compact carrier 
and with compactly supported measured laminations. 
The {\em carrier} $\Carr(\mu)\subset \Sigma$ of a current $\mu$ on $\Sigma$ is the closed subset 
obtained by projecting to $\Sigma$ all the points lying on the geodesics in the support of $\mu$. 
For a compact subset $K\subset \Sigma$,  let %$\calC_K(\Sigma)$ be the set of geodesic currents $\mu$ with carrier contained in $K$, 
\bqn
\calC_K(\Sigma):=\{\mu\in\calC(\Sigma):\,\Carr(\mu)\subset K\}
\eqn
and let 
\bqn
\MLc(\mathring \Sigma):=\{c\in\calC(\Sigma):\, i(\mu,\mu)=0\text{ and }\Carr(\mu)\subset\mathring\Sigma\text{ is compact}\}.
\eqn
%$\MLc(\mathring \Sigma)$ be the space of currents on $\Sigma$ 
%that correspond to measured geodesic laminations with compact carrier in $\mathring \Sigma$. 
Observe that the condition $i(\mu,\mu)=0$ means that any $\mu\in\MLc(\mathring \Sigma)$ corresponds to a measured geodesic lamination.

It is easy to see that there is a compact set $K\subset\mathring\Sigma$ such that any simple geodesic with compact  closure in $\mathring\Sigma$
is contained in $K$, so that $\MLc(\mathring\Sigma)\subset\calC_K(\Sigma)$.  Thus we have the following:
\begin{Theorem}\label{thm_intro:positive systole}   
Assume that $\mathring\Sigma$ is not the thrice punctured sphere.
Let $K\subset \mathring\Sigma$ be a compact subset such that 
$\MLc(\mathring\Sigma)\subset\calC_K(\Sigma)$.
Let $\mu$ be a geodesic
  current on $\Sigma$.
Then the following are equivalent:
\begin{enumerate}
\item\label{item:positive systole(1)} $\Syst(\mu)>0$;
\item the function $\lambda\mapsto i(\mu,\lambda)$ does not vanish on $\MLc(\mathring\Sigma)\smallsetminus\{0\}$;
\item the function $\lambda\mapsto i(\mu,\lambda)$ does not vanish on
 $\calC_K(\Sigma)\smallsetminus\{0\}$;
\item\label{item:positive systole(2)} Every geodesic recurrent in $\mathring\Sigma$ intersects transversely some geodesic in the support of $\mu$.
\end{enumerate}
\end{Theorem}
We say that a geodesic $g\in\calG(\H)$ is recurrent in $\mathring \Sigma$ if the projection map from $g$ into $\mathring \Sigma$ is not proper.
If $\Sigma$ has no cusps, then property (\ref{item:positive systole(2)}) in Theorem~\ref{thm_intro:positive systole} 
is equivalent to the current $\mu$ being binding in the sense of \cite[Definition~3.1]{EM18}.
If $\Sigma$ is compact and $\partial\Sigma=\varnothing$, every geodesic is recurrent.

\begin{Example} \be 
\item Unlike the case where $\Sigma$ is closed, in general there are
geodesic currents with positive systole and whose support is a lamination: 
for example if in Example~\ref{ex:1} we take the lamination $\Lambda$ such that the complementary regions are ideal polygons, 
the corresponding current $\mu$ will have positive systole.
\item The Patterson-Sullivan current (see Example~\ref{ex: basic currents} (3)) has positive systole.
\ee
\end{Example}
For the next result we endow the space $\calC(S)$ of geodesic currents on $S$, and hence $\calC(\Sigma)$, with the weak*-topology 
coming from the dual of the topological vector space of continuous compactly supported functions on $\calG(\H)$. 
The space $\P\calC(\Sigma):=\R_{>0}\backslash (\calC(\Sigma)\smallsetminus\{0\})$ of projectivized geodesic currents is 
then endowed with the quotient topology and, as such, it is compact metrizable (see Proposition~\ref{prop:proper}).

\begin{Corollary}\label{cor_intro:MCG} 
\begin{enumerate} 
\item The systole function $\Syst:\calC(\Sigma)\to[0,\infty)$ is continuous.
\item For every current $\mu$ on $\Sigma$, with $\Syst(\mu)>0$ and every compact subset $K \subset \mathring\Sigma$, there are constants $0 < c_1 \le c_2 < + \infty$ such that
\bqn
c_1 \,\ell_{\mathrm {hyp}}(c) \le i(\mu,c) \le c_2 \,\ell_{\mathrm {hyp}}(c)
\eqn
for every closed geodesic $c \subset K$. Here $\ell_{\mathrm {hyp}}$ denotes the hyperbolic length.
\item The set $\Omega = \{[\mu] \in {\P} \calC(\Sigma): \Syst(\mu) > 0\}$ is open in $\P\calC(\Sigma)$ 
and the mapping class group of $\Sigma$ acts properly discontinuously on it.
\end{enumerate}
\end{Corollary}

\begin{Remark}\label{rem1} 
If $S=\Sigma$ is a compact surface, in \cite{Martone_Zhang} the authors establish interesting systolic inequalities for \emph{period minimizing} currents with full support,
that is, currents with full support such that 
\bq\label{eq:MZ}
\{c \subset S:\, \text{ closed geodesic with }i(\mu, c) \le T\}
\eq 
is finite for every $T$. 
In fact our results imply that the first condition is redundant,
that is, every current with full support is period minimizing.
Indeed by Theorem~\ref{thm_intro:positive systole} a current with full support has necessarily positive systole.
By Corollary~\ref{cor_intro:MCG}(3) with $K=\Sigma$ this implies the finiteness of the set in \eqref{eq:MZ}.
\end{Remark}

\medskip
Assume now that $\mu$ is a geodesic current on $\Sigma$ such that $i(\mu,c)>0$ for all closed geodesics $c\subset\mathring\Sigma$.
Then $\calE_\mu$ is the set of boundary components of $\Sigma$ and $\mathring\Sigma$ is the unique complementary region.
According to \eqref{eq:decomp} we have a decomposition
\bqn
\mu=\mu_{\calE_\mu}+\mu_{\mathring\Sigma}
\eqn
and the next result gives the structure of the geodesic current $\mu_{\mathring\Sigma}$
if the systole of $\mu$ vanishes.

\begin{Theorem}\label{thm_intro:irred and syst 0}   
Let $\mu$ be a geodesic current on $\Sigma$
% as in Theorem~\ref{thm_intro:dec}(2). 
with positive intersection with every closed geodesic $c\subset\mathring\Sigma$. 
Then the following are equivalent:
\begin{enumerate}
\item\label{item:irred and syst 0(1)} $\Syst(\mu)=0$;
\item\label{item:irred and syst 0(2)} 
$\mu_{\mathring\Sigma}$ corresponds to a measured lamination,
 compactly supported in $\mathring\Sigma$, minimal 
and surface filling.
\end{enumerate}
\end{Theorem}

We say that a geodesic lamination is \emph{surface filling} if 
every complementary region is either an ideal polygon, or an ideal polygon bounding either a
boundary geodesic or a cusp.

\begin{Remark}  The relation between our results and the ones in \cite{EM18} is the following.
	If $\Sigma$ is compact, it follows from Theorem~\ref{thm:1}(1)  that a current $\mu$ of {\em full hull} in the sense of \cite{EM18}
	is one such that $i(\mu,c)>0$ for every closed geodesic $c\subset\mathring{\Sigma}$.
	Then Theorem~\ref{thm_intro:positive systole} and Theorem~\ref{thm_intro:irred and syst 0} 
	show that the dichotomy ``$\Syst(\mu)>0$" or ``$\Syst(\mu)=0$''
	gives for $\mu$ of full hull the dichotomy ``$\mu$ is binding'' or ``$\mu$ is a measured lamination".
	That this is a dichotomy for currents of full hull is the content of \cite[Theorem~3.19]{EM18}.
\end{Remark}

In general we  consider the decompostion 
\bq\label{eq:V}
\Sigma=\bigcup_{v\in\calV_\mu} \Sigma_v\,.
\eq  
as a union of subsurfaces with geodesic boundary  induced by the set $\calE_\mu$   closed $\mu$-short solitary geodesics. 
For any such a subsurface $\Sigma_v$ %let $\mathcal G_v\subset (\partial \H^2)^{(2)}$ be the set of geodesics whose projection lies in the interior $\mathring\Sigma_v$ of $\Sigma_v$, 
and for a geodesic current $\tau$ let 
$$\Syst_{\Sigma_v}(\tau)=\inf\{i(\tau, c)|\,c\subset\mathring\Sigma_v \text{ closed geodesic }\}\,.$$
%while $\Syst(\tau)$ refers to the analogous quantity with the infimum taken over all the geodesics in $\Sigma$. 
Combining Theorem \ref{thm_intro:positive systole} and Theorem \ref{thm_intro:irred and syst 0} we  deduce

\begin{Corollary}\label{cor_intro:dec}
	Let $\mu$ be a geodesic current on a complete hyperbolic surface of finite area $\Sigma=\Gamma\backslash\H^2$,  
	and let  $\calV_\mu$ be as in \eqref{eq:V}.
	We have
	$$\mu=\sum_{v\in \calV_\mu} \mu_{\Sigma_v}+\sum_{c\in \calE_\mu}\lambda_c\delta_{c}\,,$$
	where %$\mu_v$ is supported in $\mathcal G_v$ and 
	$\delta_{c}$ is the geodesic current associated to the closed geodesic $c$. 
	Furthermore, for every $v\in \calV_\mu$ for which $\mu_{\Sigma_v}\neq 0$ precisely one of the following holds:
	\begin{enumerate}
		\item either $\Syst_{\Sigma_v}(\mu_v)>0$,
		\item or $\mu_v$ is the geodesic current associated to a measured lamination %$\wt\Lambda$
		compactly supported in $\mathring\Sigma_v$ and  surface filling in $\Sigma_v$.
	\end{enumerate}
\end{Corollary}

\subsection{Positive systole and a domain of discontinuity in the Weyl chamber length boundary}
\label{s:Intro_TP-boundary}
%\subsection{Domains of discontinuity in the Weyl chamber length  boundary}
Let $\Gamma<\PSL(2,\R)$ be a (not necessarily torsion-free) cocompact
lattice and let $G$ be either $\PSL (n,\R)$ or $\Sp(2m, \R)$. 
Let $\calX(\Gamma,G)$ be the space of $G$-conjugacy classes of representations of $\Gamma$ in $G$
that are Hitchin if $G = \PSL (n,\R)$ or maximal if $G = \Sp(2m, \R)$ (see \S~\ref{sec:7}).
Our objective is to apply our results on currents to the study 
of the action of the mapping class group on the Weyl chamber length  boundary
$\partial\calX(\Gamma,G)$.
Recall (see \cite{Parreau12}) that $\partial\calX(\Gamma,G)$ 
is a compact subset of the space $\P(\mathfrak{C}^\Gamma)$ 
of projective classes of functions from $\Gamma$ to a closed Weyl chamber
$\mathfrak{C}$ of $G$,
and that
a diverging sequence $([\rho_k])_k$ in $\calX(\Gamma,G)$ converges to the
projective class $[L]$ of a nonzero function $L:\Gamma\to \mathfrak{C}$ if and
only if  $[\lambda\circ\rho_k]$ converges to $[L]$ in $\P(\mathfrak{C}^\Gamma)$, where
$\lambda\colon G\to \mathfrak{C}$ is the Jordan projection, see \eqref{eq:wcvtv}.
% the projective classes of the nonzero functions $L$ that are limits of
% sequences $\frac{1}{\lambda_k}\nu\circ\rho_k$ where
% $\nu\colon G\to \mathfrak{C}$ is the Jordan projection, $(\lambda_k)$
% is a diverging sequence of positive reals and $[\rho_k]$ is a
% diverging sequence in $\calX(\Gamma,G)$.
%
Let $\|\,\cdot\,\|$ denote a Weyl group invariant norm on the Cartan subalgebra $\mathfrak a\supset\mathfrak{C}$
and define the systole of a function $L:\Gamma\to \mathfrak{C}$ by
\bqn
\Syst(L):=\inf_{\stackrel{\gamma\in\Gamma}{\gamma\text{ hyperbolic}}}\|L(\gamma)\|\,.
\eqn
Observe that the positive systole subset 
\[\Omega(\Gamma,G):=\{[L]\in\partial\calX(\Gamma,G):\,\Syst(L)>0\}\]
of $\partial\calX(\Gamma,G)$ is well-defined and independent of the choice of the norm $\|\,\cdot\,\|$.

The next result is a consequence of Corollary~\ref{cor_intro:MCG} and \cite[Theorem~1.1]{Martone_Zhang}:

\begin{Corollary}
  \label{cor_intro:pd}  
Let $\Gamma<\PSL(2,\R)$ be a cocompact lattice and let $\calX(\Gamma,G)$ be
the character variety of representations of $\Gamma$ in $G$ that are either Hitchin or maximal.
\begin{enumerate}
\item 
  $\Omega(\Gamma,G)$ is an open subset of $\partial\calX(\Gamma,G)$.

\item 
For every $[L]\in\Omega(\Gamma,G)$ there are constants $0<c_1\leq c_2<\infty$
such that, for every hyperbolic element $\gamma\in\Gamma$,
\bqn
c_1\ell_{\mathrm {hyp}}(\gamma)\leq \|L(\gamma)\|\leq c_2\ell_{\mathrm {hyp}}(\gamma)\,,
\eqn
where $\ell_{\mathrm {hyp}}(\gamma)$ is the translation length of $\gamma$ in $\H$.
\item Assume that $\Gamma$ is torsion-free.  Then the mapping class group $\mcg(S)$ of $S:=\Gamma\backslash\H$
acts properly discontinuously on $\Omega(\Gamma,G)$.
\end{enumerate}
\end{Corollary}

Note that if $G=\PSL(2,\R)$ and $\Gamma$ is torsion-free,
$\partial\calX(\Gamma,G)$ is the Thurston boundary of the Teichm\"uller
space of $S=\Gamma\backslash \H$ and it is well-known that $\Syst(L)=0$ for
every $[L]\in\partial\calX(\Gamma,G)$.
It is a striking fact that it is not anymore the case when $G$ has higher rank.
% It is well known  that $\Omega(\Gamma,\Sp(2,\R))=\Omega(\Gamma,\PSL(2,\R))=\varnothing$.
% This is however not the case anymore in high rank.  To see this recall that 
% if $G=\PSL(2,\R)^n$, the character variety $\calX(\Gamma,G)$ is the $n$-fold Cartesian product of Teichm\"uller spaces.
% In the forthcoming article \cite{BIPP3} we identify $\partial\calX(\Gamma,G)$ for $G = \SL(2,\R)^n$ 
% and conclude that $\Omega(\Gamma,\Sp(2m,\R)) \not= \varnothing$ for all $n \geq 2$. 
% Incidentally, we show there also that for $G = \PSL(2,\R)^2$, $\Omega(\Gamma,G)$ is precisely 
% the Teichm\"uller space of semi-translation structures on $S$.

% Here we treat another case in which $\Omega(\Gamma,G)\neq\varnothing$.

\begin{Corollary}\label{cor:split}
  Let $S=\Gamma\backslash \H$ be any compact hyperbolic surface,
  and $G=\PSL(n,\R)$ with $n\geq 3$ or $G=\Sp(2m,\R)$ with $m\geq 2$.
Then the positive systole subset $\Omega(\Gamma, G)$ is non-empty.
% $\Omega(\Gamma,\PSL(n,\R))$ is a non-empty open set of discontinuity
% for the $\mcg(S)$-action on the Weyl chamber length  boundary $\partial\calX(\Gamma,\PSL(n,\R))$
% of the $\PSL(n,\R)$-Hitchin component of $S$.
\end{Corollary}

It is a natural question whether length functions if
$\Omega(\Gamma,G)$ correspond to some kind of geometric structures on
the surface $S=\Gamma\backslash \H$.
In the case where $G=\PSL(2,\R)\times\PSL(2,\R)$, one can show using
\cite{GaMa91} that $\Omega(\Gamma,G)$ corresponds to 
% or holomorphic quadratic differentials
the Teichm\"uller space of semi-translation structures $S$. 
%\marginpar{BP I am not sure if we can really prove this: I don't think that we know how to show that the map that to a flat surface associates its marked length with respect to the $\ell ^1$ norm is injective\\ {\color{red} AI+MB This is not what we are saying because our length functions are Weyl chambers valued}} 
We will study this case in further details in a forthcoming paper.

%%% Entropy
Another interesting feature of the positive systole subset is its
relationship with entropy.  This uses \cite[Corollary~1.5]{Martone_Zhang}.

\begin{Corollary}\label{cor_intro:1.10}
Let $S=\Gamma\backslash\H$ be a compact hyperbolic surface
and $G=\PSL(n,\R)$ or $G=\Sp(2m,\R)$.
  Assume that $([\rho_k])_k$ is a sequence in $\calX(\Gamma,G)$
  converging to a point of  $\Omega(\Gamma,G)$. Then we have for the
  entropy $h(\rho_k)$ of $\rho_k$ :
  \[\lim_{k\to\infty}h(\rho_k) =0 \;.\]
\end{Corollary}

The first examples of such sequences were obtained by X. Nie \cite{Nie} for $\PSL(3,\R)$
and by T. Zhang in \cite{Tengren1} for $\PSL(n,\R)$ (see also \cite{Tengren2}). 
Examples~\ref{ex-334} and \ref{ex-pqr} below satisfy this property.

Corollary \ref{cor:split}  is in fact a consequence  of next theorem (Theorem~\ref{thm_intro:triangle}) 
concerning the Weyl chamber length  boundary of the Hitchin component of a hyperbolic triangle group
\bqn
\Delta:=\Delta(p,q,r)=\left<a,b:\, a^p=b^q=(ab)^r=e\right>\,.
\eqn
%is a hyperbolic triangle group and $G=\SL(3,\R)$.  
% Recall from \cite{Choi_Goldman_05} that the $\SL(3,\R)$-Hitchin component\marginpar{$3\to n$} of $\Delta$ is not reduced to a point
% if and only if $\min\{p,q,r\}\geq3$, in which case it is homeomorphic to $\R^2$.
% In particular it is immediate from the definition of \eqref{eq:ThP} that $\partial\calX(\Delta,\SL(3,\R))\neq\varnothing$.

\begin{Theorem}\label{thm_intro:triangle}  Let $\Delta$ be a hyperbolic triangle group and $G=\PSL(n,\R)$ 
or $G=\Sp(2m,\R)$. Then for every $[L]\in\partial\calX(\Delta,G)$
\bqn
\Syst(L)>0\,.
\eqn
\end{Theorem}

% To see why Corollary~\ref{cor:split} is an immediate consequence of Theorem~\ref{thm_intro:triangle} 
% notice that $\Delta(3,3,4)$ has a subgroup of index $24$ that is torsion-free and represents a genus $2$ surface.
% %The second assertion follows then from the fact that 
% Since the fundamental group of any compact hyperbolic surface
% can be embedded into the fundamental group of a compact hyperbolic surface of genus $2$.

\subsection{Actions on buildings}\label{subsec:ActB}
Recall that any projectivized function $[L]$ in
$\partial\calX(\Gamma,G)$ is the $\mathfrak{C}$-valued length function
of an action of $\Gamma$ on an affine Bruhat-Tits building, that is
typically not simplicial, see \cite{Parreau12}. In the case where $[L]$ has positive
systole, that is $[L]\in \Omega(\Gamma,G)$, by Corollary \ref{cor_intro:pd}(2)
the action of $\Gamma$ is displacing in the sense of \cite{DGLM}, and in
particular orbit maps are quasi-isometric embeddings.

Explicit examples of $\Delta$-actions on a simplicial %$\tilde A_2$-
building whose length function belongs to $\partial\calX(\Delta,\SL(n,\R))$ can
be obtained as follows from representations $\rho:\Delta\to\SL(n,\R(X))$, 
where $\R(X)$ is the field of rational functions.
 Let $\K=\R[[X^{-1}]]$ be the field of Laurent series endowed with its
 canonical non-Archimedean $\Z$-valued valuation $v$ for which $v(X)=-1$.
 Let  $B_n(\K)$  be the Bruhat-Tits building of $\SL(n,\K)$.
 The $\mathfrak{C}$-valued length of $g$ in $\SL(n,\K)$ is its Jordan projection 
 \[\lambda(g)=(-v(a_1), \ldots, -v(a_n))\]
 where $a_1,\ldots, a_n$ are the eigenvalues of $g$ (see \S~\ref{sec:7}).

 % Let 
% \bqn
% \nu\colon\SL(n,\R(X))\to\mathfrak{C}
% \eqn 
% be the corresponding Jordan projection.
%
%  It follows by construction that if $\|\,\cdot\,\|_2$ denotes the Euclidean norm on $\mathfrak a$,
% then for $g\in\SL(n,\R(X))$, $\|\nu(g)\|_2$ is the translation length of $g$ acting
% on the building $B_n(\R(X))$ with respect to the CAT(0)-metric.

 \begin{Corollary}\label{cor_intro:triangle}  
 Assume that $\rho\colon\Delta\to\SL(n,\R(X))$ is a representation such that 
\be
\item[(i)] $\tr\rho(\gamma_0)\in\R(X)$ has a pole at infinity for some $\gamma_0\in\Delta$, and
\item[(ii)] for all $t\in\R$ large enough,
  the specialization $\rho_t$ of $\rho$ at $t$ is a Hitchin representation.
\ee
Then
\be
\item  $[\lambda\circ\rho]\in\Omega(\Delta,\SL(n,\R))$.

\item
%   the action on $B_n(\R(X))$ of every
%torsion-free subgroup $\Gamma<\Delta$ of finite index is displacing in the sense of \cite{DGLM}.  
% and there exists $0<c_1\leq c_2<\infty$ such that for every $\gamma\in\Delta$,
% \bq\label{eq:displ}
% c_1\ell_{\mathrm {hyp}}(\gamma)\leq\|\nu(\rho(\gamma))\|_2\leq c_2\ell_{\mathrm {hyp}}(\gamma)\,.
% \eq
Any $\Delta$-orbit in $B_n(\R(X))$ is a quasi-isometric embedding.
% \item If $\Syst(\rho_t)$ is the systole of the length function $\gamma\mapsto\|\nu\circ\rho_t(\gamma)\|$, then
% \bqn
% \lim_{t\to\infty}\frac{\Syst(\rho_t)}{\ln t}>0\,.
% \eqn
 \ee
\end{Corollary}

%In the above corollary, $\Syst(\rho_t)$ refers to the systole of the length function $\gamma\mapsto\|\nu\circ\rho_t(\gamma)\|$,
%where, by an abuse of notation, $\nu$ also refers to the Jordan projection $\SL(n,\R)\to\mathfrak{C}$.

% \begin{Remark}
% \be
% \item The property \eqref{eq:displ} shows that the action on $B_n(\R(X))$ of $\Delta$ is displacing in the sense of \cite{DGLM}.  
% \item  If $[\rho]\in\calX(\Gamma,\PSL(2,\R))$, it is well known that its entropy $h(\rho)\equiv1$.
% Applying  \cite[Corollary~1.5]{Martone_Zhang} we deduce that for the entropy $h(\rho_t)$ of the representations in Corollary~\ref{cor_intro:triangle},
% \bqn
% \lim_{t\to\infty}h(\rho_t)=0\,.
% \eqn
% Examples of sequences of  $\SL(3,\R)$-Hitchin representations with this property were first obtained in \cite{Zhang15}.
% \ee
% \end{Remark}

We now give explicit examples of representations of triangle groups
satisfying the hypotheses of Corollary~\ref{cor_intro:triangle}.

\begin{Example}
  \label{ex-334}
  For $\Delta=\Delta(3,3,4)$,
%For an explicit example satisfying the hypotheses of
%Corollary~\ref{cor_intro:triangle},
  define
\bqn
\ba
\rho(a)
&=\bpm
 0&0&1\\1&0&0\\0&1&0
\epm
\\
\rho(b)
&=\bpm
1&2-X+X^2&3+X^2\\0&-2+2X-X^2&-1+X-X^2\\0&3-3X+X^2&(-1+X)^2
\epm\,.
\ea
\eqn
According to the main result of \cite{LRT},  $\rho_t$ belongs to the $\SL(3,\R)$-Hitchin component of $\Delta=\Delta(3,3,4)$ for every $t\in\R$.
In addition one verifies that 
\bqn
\tr(\rho(a^{-1}b))=2X^2-3X+6\,.
\eqn

\end{Example}

\begin{Example} 
  \label{ex-pqr}
The following example is due to Goldman, \cite[\S 6]{Goldman_boulder}.
  Let $\Delta=\Delta(p,q,r)$ with $\min(p,q,r) \geq 3$.
%For an explicit example satisfying the hypotheses of
%Corollary~\ref{cor_intro:triangle},
  Consider the following matrix with coefficients in $\R(X)$, where $\epsilon_s:=\cos\left(\frac{2\pi}{s}\right)$ for $s>0$,

  \bqn
  B(X)
=
\bpm
1&-X^{-1}\epsilon_p&-\epsilon_q\\
-X\epsilon_p& 1& -\epsilon_r\\
-\epsilon_q & -\epsilon_r& 1
\epm
\eqn
Define $\rho(r_i): =-\Id+2B(X)e_i {}^te_i$ where $e_1,e_2,e_3$ are the
canonical basis vectors of $\R(X)^3$. Then
\bqn
\ba
\rho(a)
&=\rho(r_1)\rho(r_2)
=\begin{pmatrix}
	-1&-2X^{-1}\epsilon_p&0\\
	2X\epsilon_p&-1+4\epsilon_p^2&0\\
	2\epsilon_q&2\epsilon_r+4X^{-1}\epsilon_p\epsilon_q&1
  \end{pmatrix}\\
\rho(b)
&=\rho(r_2)\rho(r_3)
=\begin{pmatrix}
	1&2X^{-1}\epsilon_p&2\epsilon_q+4X^{-1}\epsilon_p\epsilon_r\\
	0&-1&-2\epsilon_r\\
	0&2\epsilon_r&-1+4\epsilon_r^2
  \end{pmatrix}
\ea
\eqn
define a representation
$\rho: \Delta \to \SL(3,\R(X))$,
whose specialization $\rho_t$ at all $t>0$  is Hitchin.
In addition a computation gives 
\bqn
\tr(\rho(a^{-1}b))=
  8\epsilon_p\epsilon_q\epsilon_r(X+X^{-1})+ 16\epsilon_p^2\epsilon_r^2+4\epsilon_q^2-1  
\eqn
hence $\rho$ satisfies the hypotheses of
Corollary~\ref{cor_intro:triangle}, provided $p,q,r\neq4$.
\end{Example}

\subsection{Outline of the paper}
After some preliminaries on currents in \S~\ref{sec:2}, we study in \S~\ref{sec:3} a general set of geodesics $A$ in $\calG(\H)$ 
and associate to it a lamination using the intersection graph of $A$. If $A$ is invariant under $\Gamma$, we deduce, 
using the structure of complementary regions of laminations in $\Sigma$, general results from which Theorem~\ref{thm:1} follows. 

The main goal of \S~\ref{sec:4} is to show that the systole of a current can be computed using only simple closed geodesics, 
provided $\mathring{\Sigma}$ is not the thrice punctured sphere. 
To this end we associate to any geodesic current on $\H$ a pseudo-distance on $\H$ 
that is a modification of a pseudo-distance introduced by Glorieux \cite{Glorieux} and which has the advantage of being additive on colinear triples of points. 
When $\mu$ is a geodesic current on a hyperbolic surface $S=\Gamma\backslash  \H$, 
this pseudo-distance leads to a pseudo-length for paths and closed curves on $S$ 
and the main point consists then in showing a Length-Shortening-Under-Surgery Lemma (Lemma~\ref{lem4.9}).
This is essential in the proof of Theorem~\ref{thm_intro:positive systole} and Theorem~\ref{thm_intro:irred and syst 0}.

In \S~\ref{sec:5} we deal with currents of positive systole. 
The main point in the proof of Theorem~\ref{thm_intro:positive systole} consists in showing 
that positive systole currents do not admit $\mu$-short recurrent geodesics. 
This is shown in Proposition~\ref{prop5.1} using the classical Closing Lemma.

In \S~\ref{sec:6}, we prove Theorem~\ref{thm_intro:irred and syst 0}, 
which follows essentially from a study of geodesic laminations consisting of $\mu$-short geodesics. 

In \S~\ref{sec:7}  we apply the results on currents to the study of the Weyl chamber length  boundary
of the Hitchin or maximal components of a cocompact lattice $\Gamma<\PSL(2,\R)$.
Beside the results of \cite{Martone_Zhang}, an essential input is Theorem~\ref{thm:7.2}
establishing that for a hyperbolic triangle group $\Delta$, 
$\Delta$-invariant non-vanishing currents have positive systole.
The basis for the construction of the explicit examples in \S~\ref{subsec:ActB} is Corollary~\ref{cor_intro:triangle},
which relies on Theorem~\ref{thm:7.2} as well as on Puiseux's theorem 
on the representability of elements of a specific real closure of $\R(X)$ by convergent Puiseux series.

\section{Preliminaries on currents}\label{sec:2}
\label{sec:preliminaries}
% We let $\partial\H$ be the boundary of $\H$ which is a circle endowed 
% with its orientation and corresponding cyclic ordering of triples of
% points. 

% We identify the space of oriented geodesics in $\H^2$ with the space $(\bH^2)^{(2)}$ of pairs of distinct  points. 

%%% def current
A  \emph{geodesic current} is a  positive Radon measure on 
the space $\calG(\H)$ of unoriented, unparametrized, geodesics in 
$\H$. The topology on $\calG(\H)$ is obtained by identifying
this space with the quotient by the flip $\sigma(a,b)=(b,a)$
of the space $(\bH)^{(2)}:=(\bH\times \bH) \smallsetminus \Delta$ of pairs of distinct
points in the boundary $\bH$ of the hyperbolic plane $\H$.
Via this identification we will think of a geodesic current as
a $\sigma$-invariant positive Radon measure on the locally compact space $(\bH)^{(2)}$.

Let $\Gamma<\PSL(2,\R)$ be a torsion-free discrete subgroup and 
$\Sf=\Gamma\backslash \H$ be the quotient hyperbolic surface.
We denote by $\pi:\H\to\Sf$ the corresponding projection.

A  \emph{geodesic current on $\Sf$} is a $\Gamma$-invariant geodesic current. 

\begin{Examples}
\label{ex: basic currents}
The following examples of geodesic currents will play an important role in the rest of the paper:
\noindent
\begin{enumerate} 
\item The \emph{Liouville current} $\calL$ is the unique (up to positive scaling)  
$\PSL(2,\R)$-invariant Radon measure on $(\bH)^{(2)}$.  
It is of course $\Gamma$-invariant for every subgroup $\Gamma<\PSL(2,\R)$.
% In other words it is the only geodesic current associated to all
% hyperbolic surfaces.  
% See Example~\ref{ex:Liouville} for an explicit
% description of $\calL$. 

\item Let $c$ be a geodesic in $\Sf$. 
% commA : preciser def
Then the set of  lifts of $c$ to $\H$ is a $\Gamma$-orbit in $\calG(\H)$, 
which is  discrete in $\calG(\H)$ if and only if $c$ is closed as a subset of $\Sf$.
Then the sum $\delta_c$ of the Dirac masses along this orbit is a
geodesic current on $\Sf$.
%
% For $\g\in\G$ hyperbolic, we set
% \bqn\delta_\g:=\sum_{\eta\in\G/\<\g\>}\delta_{\eta(\g_-,\g_+)}.\eqn
% Then $\delta_\g$ is a geodesic current whose carrier is the closed geodesic $c\subset \Sigma$ associated to $\g$. We will sometimes denote $\delta_\g$ by $\delta_c$.

\item Let $\D=\{z\in\C: |z|<1\}$ be the unit disk model of $\H$ and $\Gamma<\PSU(1,1)$ be a discrete subgroup. 
For $\delta\geq 0$ a \emph{$\delta$-density} on $\partial\D$ is a probability measure $\nu$ 
such that $\nu(f\circ \gamma^{-1})=\nu(j_\gamma^\delta \cdot f)$ where 
\bqn
j_\gamma(\xi)=\frac 1{|a\xi+b|^2},\quad \gamma=\begin{pmatrix}a &b\\\ov b&\ov a\end{pmatrix}.
\eqn
For instance the round measure 
\bqn
\lambda(f)=\frac 1{2\pi}\int_{0}^{2\pi}f(e^{i\theta})d\theta
\eqn
is a 1-density for $\PSU(1,1)$. Given a $\delta$-density $\nu$, the measure $\mu$ on $(\partial\D)^{(2)}$ given by 
\bqn
\mu(f)=\int\int\frac{ f(\xi,\eta)}{|\xi-\eta|^{2\delta}}d\nu(\xi)d\nu(\eta)
\eqn
is then a $\Gamma$-invariant current. 

When $\Gamma$ is finitely generated, and $\delta$ is the critical exponent of $\Gamma$, 
there is a unique $\delta$-density $\mu$ on $\partial \D$; 
its support is precisely the limit set $\Lambda$ and the corresponding measure $\mu_{\rm PS}$ is the \emph{Patterson--Sullivan current}. 
It is thus a current on the convex core $\Sigma$ of $S$ 
and every recurrent geodesic in $\mathring \Sigma$ intersects transversely a geodesic in the support of $\mu_{\rm PS}$. 
It follows from Theorem~\ref{thm_intro:positive systole} that $\Syst(\mu_{\rm PS}) > 0$ and $\mu_{\rm PS}$ satisfy the conclusion of Corollary~\ref{cor_intro:MCG}~(2).
\end{enumerate}
\end{Examples}

%%% Def intersection
Recall that if $g,h\in\calG(\H)$ are two geodesics, their intersection number $i(g,h)$ is defined as
\bqn
i(g,h):=
\begin{cases}  
0&g,h\text{ are disjoint or coincide}\\ 
1&g,h\text{ intersect transversally.}
\end{cases}
\eqn
Then the \emph{intersection} $i(\mu,\nu)$ of two currents $\mu, \nu$ on $\Sf$ is defined as follows (see \cite{Bon88-lam} or \cite{Martelli}).
Let 
\bqn
\calDG(\H):=\{(g,h)\in\calG(\H)\times\calG(\H):\,i(g,h)=1\}\,;
\eqn 
% in exactly one point. 
then  $\PSL(2,\R)$ acts properly on the open set $\calDG(\H)$ and so does $\Gamma$. 
% indeed the map that associate to a pair $(g,h)$ in $\mathcal D$ the intersection $g\cap h$ gives a $\PSL(2,\R)$-invariant projection of $\mathcal D$ to $\H$. 
% 
%restrict the product $\mu\times \nu$ to $\calDG(\H)$ and 
The intersection  $i(\mu,\nu)$ is then  the $(\mu\times\nu)$-measure
of any Borel fundamental domain for the $\Gamma$-action on $\calDG(\H)$.
We will often denote $i(\mu,\delta_c)$  by $i(\mu,c)$.

\begin{Examples}
\begin{enumerate}
\item Given two distinct closed geodesics $c,c'$ in $\Sf$, the intersection
  $i(\delta_c,\delta_{c'})$ is the minimal geometric
  % ?
  intersection
  number between two closed curves in the free homotopy classes represented  by $c$ and $c'$. Instead  $i(\delta_c,\delta_{c})$ is the number of self intersections of the geodesic $c$, so for example $i(\delta_c,\delta_c)=0$ if and only if $c$ is simple.
% We will often denote it by $i(c,c')$.
\item If $c$ is a closed geodesic in $\Sf$ and $\ell_{\rm hyp}(c)$ is its 
  hyperbolic length,  then
% For an appropriate normalization of the Liouville current  $\calL$ 
we have 
\bqn
i(\calL,\delta_c)=\ell_{\rm hyp}(c)\;.
\eqn 
\end{enumerate}
\end{Examples}

%%% Topology on the space of currents
The set $\calC(\Sf)$ of geodesic currents on $\Sf$
is a convex cone in the dual %$\mathcal C_{cc}((\bH^2)^{(2)})^*$ 
of the space of compactly supported functions on $\calG(\H)$; 
the latter is provided with the topology of inductive limit of Banach
spaces and $\calC(\Sf)$ will be equipped with the corresponding weak*
topology. 
Given a geodesic current $\mu$ on $\Sf$, its support
$\supp(\mu)\subseteq \calG(\H)$ is a closed $\Gamma$-invariant subset;
as mentioned in the introduction, we call the {\em carrier} of $\mu$ and denote by $\Carr(\mu)$
the closed subset of $S$ consisting of the projection to $\Sf$ of the union of all points on all geodesics in $\supp(\mu)$.  
For $F\subset \Sf$, we denote by $\calC_F(\Sf)$ 
the subset of geodesic currents on $\Sf$  with  carrier included in
$F$.

It is straightforward to verify that if $\mu$ is any current on $S$ and $\nu$ has compact carrier, then
$i(\mu,\nu)<+\infty$.
% We denote by $\Ccc(\Sf)$ 
% the set of geodesic currents with compact carrier.

%%% Geom finite case
%In this article we will be mostly interested in the case in which
%$\Gamma$ is finitely generated, that is $\Sf$ is geometrically
%finite. 
%In this case $\Sf$ has finitely many cusps and finitely many expanding
%ends (funnels); if $C\subset\H$ denotes the closed convex hull of the
%limit set $\Lambda_\Gamma\subset\bH$ of $\Gamma$, then
%$\Sigma:=\Gamma\backslash C\subset S$ is a finite area hyperbolic surface with
%geodesic boundary; 
%%%commA: core of $\Sf$ ??
%a geodesic current on $\Sigma$
%is a geodesic current on $\Sf$  with carrier included in $\Sigma$.

\begin{Example}
The geodesic current  $\delta_c$ on $\Sf$ 
from Example~\ref{ex: basic currents}(2) is a
geodesic current on $\Sigma$ if and only if $c$ is either a closed
geodesic or an \emph{ideal} geodesic,
that is a geodesic connecting two cusps of $\Sigma$.
This is the case in Figure~\ref{fig:1}.
\end{Example}

One of the most fundamental facts concerning the intersection is the following continuity property due to Bonahon:

\begin{Theorem}[{\cite[\S~4.2]{Bon86}}]
For every compact subset $K\subset \Sf$, the intersection
\bqn
i\colon\mathcal \calC(\Sf)\times \calC_K(\Sf)\to \R
\eqn
is continuous.
\end{Theorem}

%%% Generic pencils
This continuity relies on the following crucial technical point that we will use
in this paper and that can be found for example in \cite[Proposition~8.2.8]{Martelli}.
It rules out a situation in which a geodesic current may have a ``one-sided'' atom.
Recall that a \emph{pencil} is a subset of $(\bH)^{(2)}$ of the form $\{p\}\times B$ where $B\subset\bH$
is a Borel subset not containing $p$. The lemma holds as stated for any hyperbolic surface $S = \Gamma \backslash \H$. If $S$ is geometrically finite, then a point in the limit set is conical if and only if it is not a cusp.

\begin{Lemma}\label{lem:pencils}
Let $\mu$ be a geodesic current on $\Sf$. 
Assume that $p\in\bH$ is a conical limit point and the pencil
$\{p\}\times B$ does not contain the axis of a hyperbolic element. Then
\bqn
\mu(\{p\}\times B)=0 \;.
\eqn
\end{Lemma}

%\begin{Remark}
%This lemma holds as stated for any hyperbolic surface $S = \Gamma \backslash \H$.
%If $S$ is geometrically finite, then a point in the limit set is conical if and only if it is not a cusp.
%\end{Remark}
% This continuity property implies a certain number of compactness
% criteria which will be useful in \S~\ref{sec:pos}.
For our purposes we will need the following compactness property,
which can be derived from the arguments in \cite[Proposition 4]{Bon88-curr}.

\begin{Proposition}\label{prop:proper}
\be
\item Let $K\subset S$ be compact and let $\calL$ be the Liouville current on $S$. Then the set 
\bqn
\{\mu\in\calC_K(\Sf):\ i(\calL,\mu)=1\}
\eqn
is compact and hence $\P\calC_K(S)$ is compact.
\item The space $\P\calC(\Sigma)$ is compact.
\ee
\end{Proposition}

\begin{proof}  The proof of (1) follows the one in \cite[Proposition~4]{Bon88-curr}
observing that, since $\supp(\calL)=\calG(\H)$,
any non-zero current $\mu$ on $S$ has positive intersection with $\calL$.
For the second assertion replace \cite[Proposition~4]{Bon88-curr} by the following lemma.
\end{proof}

\begin{Lemma}\label{lem:2.8}  Let $\nu$ be a current on $\Sigma$ with compact carrier such that 
every geodesic in $\H$ projecting into $\mathring{\Sigma}$ intersects transversally a geodesic in the support of $\nu$.
Let $\psi:\calG(\H)\to[0,\infty)$ be continuous with compact support 
such that for every boundary component $g$ of $\Sigma$, $\psi(\tilde g)=1$
for some lift $\tilde g$ of $g$.  Then the set
\bqn
\{\mu\in\calC(\Sigma):\,\mu(\psi)+i(\mu,\nu)\leq1\}
\eqn
is compact.
\end{Lemma}

The proof of Lemma~\ref{lem:2.8} is a straightforward modification of the proof of \cite[Proposition~4]{Bon88-curr}.
The lemma implies Proposition~\ref{prop:proper}(2) by observing that the set 
\bqn
P:=\{\mu\in\calC(\Sigma):\,\mu(\psi)+i(\mu,\nu)=1\}
\eqn
is compact and the projection map $P\to\P\calC(\Sigma)$ is a continuous bijection.

%Observe that since $\supp(\calL)=\calG(\H)$, any  non-zero current $\mu$ on $\Sf$ has positive intersection with $\calL$.
%In fact, if $S$ is compact, this is nothing but Bonahon's proof that $\P\calC(S)$ is compact.

We refer for instance to \cite[\S 8.3.4]{Martelli} for the notion of \emph{measured geodesic lamination} on a general hyperbolic surface $S=\Gamma\backslash \H$, 
and the bijective correspondence between geodesics currents $\mu$ on $S$ with $i(\mu,\mu)=0$, 
equivalently such that no two geodesics in the support of $\mu$ intersect transversally, and measured geodesic laminations on $S$.

\section{Decompositions}\label{sec:3}
Let $S$ be a hyperbolic surface (not necessarily complete) and $F\subset S$ a subset.  
We denote by $\calG(F)$ the set of unoriented, unparametrized geodesics of $S$ that are contained in $F$. 
Given a geometrically finite surface $S$ with convex core $\Sigma$ and an arbitrary subset $A\subset\calG(\Sigma)$ of geodesics, 
our aim is to show how one can associate two laminations $\mathcal E_A$ and $\Lamref_A$ respectively, 
such that their complementary regions are either completely avoided by $A$ or filled by $A$ in two specific ways
(see Proposition~\ref{prop:3.1} and \ref{prop:decA}).
Applying these propositions to the support of a geodesic current on $\Sigma$ will imply Theorem~\ref{thm:1}. 
In \S~\ref{s:3.1} we start by studying the case of a subset of $\calG(\H)$, then move to geometrically finite surfaces in \S~\ref{s:3.2}, 
where we establish the main propositions.  We show in \S~\ref{s:3.3} how to deduce Theorem~\ref{thm:1}.

\subsection{The lamination associated to a subset of geodesics in $\H$}\label{s:3.1}
Given a subset $A\subset \calG(\H)$ its \emph{intersection graph}, $\Graph(A)$, 
is the graph whose vertex set is $A$ and two vertices $a, a'$ are adjacent if $i(a,a')=1$. 
We say that $A$ is {\em $i$-connected} if $\Graph(A)$ is connected; 
an {\em $i$-connected component} ($i$-cc) $A'\subset A$ is then the set of vertices of a connected component of $\Graph(A)$. 
We proceed to define the lamination associated to $A$, this relies on classical properties of convex hulls in $\H$.

For $A\subset \calG(\H)$ let $\Hull(A)\subset \H$ be the convex hull of the union of the geodesics in $A$, 
namely the intersection of all convex subsets containing $A$. 
Whenever $|A|>1$, then $\Omega:=\Hull(A)$ is in general neither open nor closed 
but its closure $\ov\Omega$ is the closed convex hull of the set $T
\subset \partial \H$ of extremities of geodesics in $A$. Each connected component of $\partial \H\setminus\ov T$ is an interval to which we can associate the geodesic connecting its endpoints;  the boundary of $\overline\Omega$ in $\H$ is the disjoint union of the set $\Delta(\Omega)$ of all such geodesics.
% the set $\Delta(\Omega)$ \marginpar{\tiny (2) Correct the definition, add a picture}
%of geodesics connecting the endpoints of the intervals 
%that are connected components of $\partial \H \backslash \overline T$. 
\begin{center}
	\begin{figure}[h]
		\begin{tikzpicture}
			\draw (0,0) circle [radius=2];

		%	\draw (-1.414,1.414) -- (1.414,-1.414);
		%	\draw (-2,0) arc (90: 37:4);
		%	\draw (0,2) arc (180: 233:4);
		%	\draw (-0.8,1.83) arc (201.5:231.6:7.5);
		%	\draw (-1.88,.7) arc (70:40:7.5);
			% geodesic arc I
		%	\draw (.5,0.05) arc (125:147.5:2);
		%	\draw (.5,0.1) node[right] {$I$};
			\begin{scope}[rotate=40]
				\draw (-1.88,.7) arc (70:40:7.5);
			\end{scope}
				\draw (-0.8,1.83) node[above] {$q$}  arc (201.5:231.6:7.5);
				\draw (-0.8,1.83) arc (206.5:235.6:7.5);
				\draw (-0.8,1.83) arc (211.5:240:7.5);	
				\draw (-0.8,1.83) arc (216.5:244:7.5);	
				\draw (-0.8,1.83) arc (221.5:247.6:7.5);
				\draw [red, thick] (-0.8,1.83) arc (226.5:251.6:7.5);	
				\draw [red, thick] (-0.8,1.83) arc (-14:-33:8.2);
		
				\draw [red, thick] (-1.85,-.65) arc (90:66:8.2);
		
			\node at (2.2,.4) {$y$};
		\node at (1.8,-1.4) {$x$};
		\node at (2.2,-0.3) {$r$};
		\node at (-2.2,-.7) {$p$};
		\end{tikzpicture}
		\caption{The set $A$ consists of the pencil $\{q\}\times (x,y)$ together with the geodesic $\{p,r\}$. In this case $T$ consists of the union of the open interval $(x,y)$ and the two points $\{p,q\}$. Its boundary has three connected components, and  $\Delta(\Omega)$ consists of the three geodesics in red.}
	\end{figure}
\end{center}

Let 
\begin{equation}\label{eqn:1}\Lambda_A:=\overline{\bigcup_{A' i\text{cc} \text{ of } A}\Delta(\Hull(A'))}\end{equation}
be the closure in the space $\calG(\H)$ of the set of the boundary geodesics of convex hulls of $i$-connected components of $A$. 
Notice that the boundary of the closure of $\Hull(A')$ is a lamination for every $i$-cc $A'$ and the gist of next proposition
is to show that $\Lambda_A$ is a lamination as well.

Given any subsets $A,B\subset \calG(\H)$ we set 
\bqn
i(A,B)=\sum_{g\in A, h\in B}i(g,h)\in[0,\infty].
\eqn
Observe that $i(A,B) > 0$ if and only if some $a \in A$ intersects transversally some $b \in B$.
Then
$$A^0=\{g\in\calG(\H): i(g,A)=0\}$$
is a closed subset of $\calG(\H)$ and  if $A\subset B$, then $B^0\subset A^0$.
 Setting $A^{00}:=(A^0)^0$, then $A\subset A^{00}$, and $A^0\cap A^{00}$ is a lamination, 
since $i(A^0\cap A^{00},A^0\cap A^{00})\leq i(A^0,A^{00})=0$.

\begin{Lemma}\label{lem:crucial} Let $h$ be a geodesic and $A'$ an $i$-cc of $A$.
If $i(h,\Delta(\Hull(A'))>0$, then $i(h,A')>0$.  As a consequence
\bqn
i(\Delta(\Hull(A')),\Delta(\Hull(A'')))=0
\eqn
for every two $i$-cc $A',A''$ of $A$.
\end{Lemma}

\begin{proof} If $i(h,a') = 0$ for every $a' \in A'$, every $a'$ is contained in one of the two closed half planes defined by $h$, 
and since $A'$ is $i$-connected, the same holds for ${\rm Hull}(A')$, implying $i(h,g) = 0$ for every $g\in\Delta(\Hull(A'))$.

Let $A',A''$ be distinct $i$-cc, $g' \in \Delta ({\rm Hull}(A'))$, $g'' \in \Delta ({\rm Hull}(A''))$ and assume $i(g', g'') = 1$. 
By the claim there is $a'' \in A''$ with $i(g',a'') = 1$ and hence $a' \in A'$ with $i(a',a'') = 1$, a contradiction. 
Thus $i(g',g'') = 0$.
%{\color{red} and $i(g',a'') = 0$ for every $a'' \in A''$}.\marginpar{AP: this seems useless ?}
\end{proof}

\begin{Proposition}\label{prop:3.1} 
Let $A\subset \calG(\H)$ be a set of geodesics and $\Lambda_A$ as defined in (\ref{eqn:1}). Then
\begin{enumerate}
\item $\Lambda_A$ is a lamination and $\Lambda_A=A^{0}\cap A^{00}$.  In particular $i(\Lambda_A,A)=0$;
\item for any complementary region $\mathcal R$  of $\Lambda_A$ one of the following holds
\begin{enumerate}
\item  either no geodesic of $A$ meets $\mathcal R$, or
\item $A_\mathcal R:=\{a\in A:\, a \text{ is contained in } \mathcal R\}$ is an $i$-connected component of $A$, and $\mathcal R=\Hull(A_\mathcal R)$. In particular every geodesic $g$ meeting $\mathcal R$ must intersect transversally some geodesic in $A_\mathcal R$.
\end{enumerate}
\end{enumerate}
 \end{Proposition}

\begin{proof}
It is immediate from Lemma~\ref{lem:crucial} that $\Lambda_A$ is a lamination and $i(\Lambda_A,A) = 0$.
We will prove that $\Lambda_A=A^{0}\cap A^{00}$ after having proven (2).

\medskip
If $\calR$ is a complementary region and $a \in A$ intersects $\calR$ then $a \subset \calR$ since $i(\Lambda_A, A) = 0$. 
Thus, the $i$-cc-component $A'$ of $a$ is formed of geodesics all in $\calR$ and hence ${\rm Hull}(A') \subset \calR$. 
Since $\Delta({\rm Hull}(A')) \subset \Lambda_A$, we conclude that ${\rm Hull}(A') = \calR$, thus proving (2).

We now complete the proof of (1).  Since  $i(\Lambda_A,A)=0$, it follows that $\Lambda_A \subset A^0$. 
If $i(\Lambda_A,A^0) > 0$, then there is an $i$-cc $A'$ with $i(\Delta({\rm Hull}(A')), A^0) > 0$  
which by Lemma~\ref{lem:crucial} would imply that $i(A',A^0)>0$,  a contradiction.
Thus $\Lambda_A \subset A^0 \cap A^{00}$. 

Conversely, if $g \in (A^0 \cap A^{00}) \smallsetminus \Lambda_A$, then since $A^0 \cap A^{00}$ is a lamination, 
there is a  complementary region $\calR$ of $\Lambda_A$ with $g \subset \calR$. 
Since $g \in A^0$, if follows from (2b) and Lemma~\ref{lem:crucial} that no geodesic of $A$ intersects $\calR$ nontrivially 
and hence $\calG(\calR) \subset A^0$. 
Since $\calG(\calR) \not= \varnothing$, the region $\calR$ must have a least four ideal vertices and 
hence there is $h \in \calG(\calR) \subset A^0$ with $i(g,h) = 1$ implying $g \notin A^{00}$, a contradiction. 
\end{proof}

\subsection{The structure of subsets of geodesics in $\Sigma$}\label{s:3.2}
Let now $S=\Gamma\backslash \H$ be a geometrically finite hyperbolic
surface and $\Sigma=\Gamma\backslash C$ its convex core. 
The covering projection $\pi:\H\to S$ induces a map $\calG(\H)\to \calG(S)$ still denoted by $\pi$. 
For $g,h\in\calG(S)$ we define $i(g,h)$ as the sum of $i(g',h')$ 
where $(g',h')$ runs through a fundamental domain for the $\G$-action on $\pi^{-1}(g)\times\pi^{-1}(h)$; 
if $g,h$ are distinct closed geodesics this recovers the usual intersection number. 
For subsets $A,B\subset \calG(S)$ we extend the definition of $i$ to $i(A,B)$ as in \S\ref{s:3.1}, and define $A^0$ analogously. 
Given $A\subset \calG(\Sigma)$ we consider the set $\calE_A$ of solitary elements among the set of closed geodesics in $A^0$, 
that is 
\begin{equation}
\label{eq:EA}
\begin{aligned}
\calE_A:=\{c\in\calG(\Sigma):\, c\text{ closed},\, i(A, c)=0\text{ and }i(c,c')=0&\\
 \quad \forall  c'\in \calG(\Sigma)\text{ with }c' \text{ closed and }i(A,c')=0&\}\,.
\end{aligned}
\end{equation}
Observe that if $A$ is the support of a geodesic current $\mu$ on
$\Sigma$, $\calE_A$ is nothing but
the set $\calE_\mu$ defined in the introduction. In general $\calE_A$
consists of simple, closed,
pairwise disjoint geodesics and contains all the boundary components of $\Sigma$.
In particular $\calE_A$ is a geodesic lamination and furthermore it  induces a partition of  $A$: 
\bqn
A=(A\cap \calE_A) \sqcup \bigsqcup\limits_{\calR}A_{\calR}
\eqn
where the disjoint union is over all complementary regions $\calR$ of $\calE_A$
% complementary components $\Sigma'$ of $\calE_A$ in $\Sigma$,
and $A_{\calR} :=A\cap\calG(\calR)$.

%%% Decomposition by special geodesics
\begin{Proposition}
\label{prop:decA}
Let $A$ be a subset of $\calG(\Sigma)$, and let $\calE_A$ be 
%the set of $(A^0\cap\text{closed},\calG(\Sigma))$-solitary geodesics
as in \eqref{eq:EA}. 
Then for every complementary region $\calR$ of $\calE_A$ precisely one of the following holds:
\begin{enumerate}
\item either no geodesic of $A$ meets $\calR$,
\item or any closed geodesic $c\subset\Sigma$
%%% def special without boundary 
%  not contained in $\partial \Sigma$ and 
meeting $\calR$ must intersect transversely some geodesic of $A_{\calR}$.
%  of $A$, contained in $\Sigma'$.
\end{enumerate}  
\end{Proposition}

The proof of Proposition~\ref{prop:decA} uses the structure of complementary regions of a compactly supported geodesic lamination $\Lambda\subset\Sigma$ in a complete finite area hyperbolic surface. Namely that the complement $\Sigma\smallsetminus\Lambda$ 
is a finite union of components of the following types \cite[Theorem I.4.2.8]{notesonnotes}:
\be
\item an ideal polygon;
\item an ideal polygon containing one cusp;
\item a totally geodesic subsurface with geodesic boundary to which one has added a \emph{crown} to some boundary geodesic (such a subsurface can possibly be reduced to a single geodesic).
\ee
\noindent A crown is an infinite cylinder bounded by a geodesing on one side and by finitely many ideal sides on the other.
We will need the following
\begin{Lemma}\label{l.crown}
	Let $c\subset \Sigma\setminus \Lambda$ be a geodesic bounding a crown $\calQ$. Any closed geodesic $c'\subset \Sigma$ intersecting $c$ intersects a leaf of $\Lambda$.
\end{Lemma}	
\begin{proof}
 Assume by contradiction that there exists a closed geodesic $c'$ that intersects $c$ but doesn't intersect any leaf of $\Lambda$. We choose intersecting lifts $g, g'$ of $c,c'$ in $\H$. The lift $\wt \calQ$ of $\calQ$ which is bounded by $g$ on one side is an infinite strip, bounded on the other side by countably many geodesics $l_i\in\wt \Lambda:=\pi^{-1}(\Lambda)$ indexed so that $l_i$ shares the endpoint $p_i\in\partial\H$ with $l_{i+1}$. 
Since $c'$ doesn't intersect any leaf of $\Lambda$, there exists $i$ such that $p_{i}$ is an endpoint of $g'$.  We denote by $\gamma\in\Gamma$ the hyperbolic element with axis $g'$, and assume without loss of generality that $p_i=\gamma_+$. Then for a sufficiently high power of $\gamma$, $\gamma^n g$ intersects $l_i$ and $l_{i+1}$. This implies that $c$ intersects the lamination $\Lambda$, a contradiction.
\end{proof}

\begin{center}
	\begin{figure}[h]
		\begin{tikzpicture}
			\draw (0,0) circle [radius=2];
		
			\draw (-0.9,1.75)  arc (36:-36:3);
                    \draw (2,0) node [right] {$p_0$} arc (220:188:3);
                    \draw(1.35,1.5) node [above right] {$p_1$} arc (-90: -130: 2);
                    \draw(0.1,2) node [above] {$p_2$} arc (-45: -125: .5);
                    \draw (2,0) node [right] {$p_0$} arc (-220:-188:3);
                    \draw(1.35,-1.5) node [below right] {$p_{-1}$} arc (90: 130: 2);
                    \draw(0.1,-2) node [below] {$p_{-2}$} arc (45: 125: .5);
                    \node at (0.5,0) {$\wt \calQ$};
                    \node at (-0.65,0) {$g$};

		\end{tikzpicture}
		\caption{The lift $\wt \calQ$ of a crown is an infinite strip bound by a geodesic $g$ on one side and countably many geodesics on the other side.}
	\end{figure}
\end{center}
	
%{\color{gray}
%A compactly supported geodesic lamination $\Lambda\subset\Sigma$ \emph{fills} a subsurface $\Sigma'\subseteq \Sigma$ with totally geodesic boundary if, for every closed curve $c\subset \mathring{\Sigma}'$, $c$ intersects some leaf of $\Lambda$. Observe that a measured geodesic lamination $(\Lambda, m)$, when regarded as a geodesic current, never fills a surface $\Sigma$, not even when the lamination $\Lambda$ does fill. A geodesic lamination $\Lambda$ is \emph{minimal} if every half ray of a geodesic belonging to $\Lambda$ is dense in $\Lambda$.  
%}

\begin{proof}[Proof of Proposition \ref{prop:decA}]  Let $B = \{ c \subset \Sigma$: $c$ is a closed geodesic and $i(c,A) = 0\}$. 
	We apply Proposition~\ref{prop:3.1} to the $\Gamma$-invariant set of geodesics $\wt{B}: = \pi^{-1} (B) \subset \calG(\H)$ 
and let $\widetilde{\Lambda} : = \Lambda_{\wt{B}}$ be the corresponding lamination. 
Set $\widetilde{\calE}_A : = \pi^{-1} (\calE_A)$. Since $\wt B\subset \wt B^{00}$, we have the following inclusion: 
$\wt{\calE}_A = \wt{B} \cap \wt{B}^0 \subset \wt{B}^{00} \cap \wt{B}^0 = \wt{\Lambda}$. 
% Since every geodesic in $\wt{B}\cap\wt{B}^0$ projects to a closed geodesic 
% and the same is not true anymore for geodesic in $\wt{B}^{00}\cap\wt{B}^0$,
% we need to show that every leaf of $\wt\Lambda$ projects to a closed geodesic.
\begin{claim}
In fact $\wt{\calE}_A = \wt{\Lambda}$.
\end{claim}
%We proceed to show  that in fact $\wt{\calE}_A = \wt{\Lambda}$.
\begin{proof}\phantom\qedhere
Note that, denoting  $\wt{A} : = \pi^{-1}(A)$,
we have $\wt{A}\subset\wt{B}^0$, hence   $\wt{B}^{00}\subset\wt{A}^0$.
Since $\wt B=\wt A^0\cap \{\text{closed geodeiscs of $\Sigma$}\}$, this implies that
 a leaf of $\wt\Lambda$ is in
$\wt{\calE}_A$ if and only if projects to a closed geodesic.
Recall now that
\bqn
\wt{\Lambda} = \overline{\bigcup\limits_{\wt{C} \text{$i$-cc of }\wt{B}} \,\Delta({\rm Hull}(\wt{C}))}\,.
\eqn
As $\wt{\calE}_A$ is closed, it is then enough to show that
$\Delta({\rm  Hull}(\wt{C})) \subset \wt{\calE}_A$ for all $\wt{C}$.

Let $\wt{C}$ be an $i$-cc  of $\wt{B}$ and assume first that $\wt{C} = \{b\}$. 
Then $b \in \wt{B}$ and $i(b,\wt{B}) = 0$, so that $b \in \wt{B} \cap \wt{B}^0 = \wt{\calE}_A$. 
Assume now that $|\wt{C}| > 1$ and let $\calQt = {\rm Hull}(\wt{C})$. Then $\calQ:=\pi(\calQt)$ is a complementary region in $\Sigma$ of the lamination $\Lambda : = \pi(\wt{\Lambda})$. {For the sake of contradiction,} assume that the boundary of $\calQt$ contains a leaf that does not close in $\Sigma$; by the structure of complementary regions \cite[Theorem I.4.2.8]{notesonnotes},
the projection of this leaf is then part of a crown bounded by a simple closed geodesic $c \subset \Sigma$; 
moreover since $|\wt{C}| > 1$, $\calQ$ contains at least a pair of intersecting closed geodesics and hence the region $\calQ$ is not reduced to a crown, 
implying $c \subset \calQ$. Let $g \subset \calQt$ be a lift of $c$. Since $i(\wt{A}, \wt{B}) = 0$, 
where $\wt{A} : = \pi^{-1}(A)$, Proposition~\ref{prop:3.1} (2) implies that no geodesic in $\wt{A}$ meets $\calQt$, hence $g \in \wt{B}$. 
If $c' \subset \Sigma$ is a closed geodesic with $i(c',c) > 0$, then % $c'$ cannot be contained in $\calQ$, since $c$ bounds a crown in $\calQ$, hence 
for a suitable lift $g'$ of $c'$ we have $i(g',g) > 0$ and thus  $g' \not\subset \calQt$ by Lemma \ref{l.crown}. 
This implies $g' \notin \wt{B}$, hence $i(g', \wt{A}) > 0$ and $i(c',A) > 0$ which shows that $c \in \calE_A \subset \Lambda$, a contradiction.
\end{proof}

Let now $\calR$ be a complementary region of $\Lambda$ and assume that
some geodesic of $A$ intersects $\calR$. 
Let $c \subset \Sigma$ be a closed geodesic intersecting $\calR$. 
Let $\calRt$ be a complementary region of $\wt{\Lambda}$ lifting $\calR$ and $g$ a lift of $c$ intersecting $\calRt$. 
Since $\wt{A}$ intersects $\calRt$ no geodesic of $\wt{B}$ can intersect $\calRt$ (Proposition~\ref{prop:3.1}(2)) and 
hence $g \notin \wt{B}$, that is $c \notin B$, hence $i(c,A) > 0$.
\end{proof}

We now combine Proposition~\ref{prop:3.1} and Proposition~\ref{prop:decA} to obtain a refined decomposition of $\Sigma$ associated with $A$ analogous to Theorem~\ref{thm:1}.

%%% With concrete definition of $\Lam_\mu$
% We consider the set $\Lam_A$ of \emph{solitary} $\mu$-short geodesics in $\Sigma$, 
% namely
% the set of $\mu$-short geodesics 
% $g\subset \Sigma$ which do not intersect
% transversely any other $\mu$-short geodesic $h\subset \Sigma$.

%%% Refined decomposition

\begin{Proposition}
\label{prop:refdecA}
Let $A$ be a subset of $\calG(\Sigma)$ and $\wt{A}=\pi^{-1}(A)$.
Let $\Lamref_{\wt{A}}=\calE_A\cup\Lam_A$, where $\Lam_{\wt{A}}$ is the lamination associated to $A$ by Proposition~\ref{prop:3.1}
and $\pi(\Lam_{\wt{A}})=\Lambda_A$.
Then  $\Lamref_A$ is a lamination and we have

\begin{enumerate}

\item 
\label{refdecA-it- in A0}
$i(\Lamref_A,A)=0$;
% The leaves of $\Lamref_A$ do not intersect transversely any geodesic
%   in $A$; 
in particular
  \[A=(A\cap \Lamref_A) \sqcup \bigsqcup_{\calR} A_{\calR} \]
where $\calR$ runs over the complementary regions of $\Lamref_A$ and $A_\calR:=A\cap\calG(\calR)$.
\item 
\label{refdecA-it- dicho}
 For every complementary region $\calR$ of $\Lamref_A$ precisely one of the following holds:
\begin{enumerate}
\item either $\calR$ does not meet any geodesic of $A$.
%  in particular $A_\calR=\varnothing$;
\item or
%% A is binding on calR 
every geodesic $g$ of $\Sigma$ meeting $\calR$ intersects transversely some geodesic in $A$. 
\end{enumerate}

\item 
\label{refdecA-it- not meeting 0-cc}
$\Lamref_A$ does not meet any of the complementary regions $\calR$ of $\calE_A$ meeting no geodesic of $A$.
  
\item 
\label{refdecA-it- 0-closed}
For every closed geodesic  $c\subset \Sigma$, if $i(c,A)= 0$, then $i(c,\Lamref_A)= 0$.
% $c$ do not intersect
% transversely any leaf of $\Lamref_A$ ; 

%% %%% Alternative version
%% Any closed geodesic  $c\subset \Sigma$ intersecting transversely some
%% leaf of $\Lamref_A$ must intersect transversely some geodesic in $A$;

\item 
\label{refdecA-it- closed leaves}
$\calE_A$ is the set of closed leaves of $\Lamref_A$.

\end{enumerate}
\end{Proposition}

\begin{proof}
We implicitly use that the assertions of Proposition~\ref{prop:3.1} hold true verbatim once
projected to $\Sigma$.

First observe that, because of Lemma~\ref{lem:crucial}, we  have $i(\calE_A,\Lam_A)=0$, hence
$\Lamref_A$ is a lamination.
The statements (\ref{refdecA-it- in A0}) and (\ref{refdecA-it- closed leaves})  are clear as  $\calE_A$ and $\Lam_A$ are included in $A^0$ by definition, 
and the closed geodesics in $\Lam_A$ are  in $\calE_A$.

The dichotomy (\ref{refdecA-it- dicho}) is true for $\Lam_A$  by Proposition~\ref{prop:3.1}, 
hence also for $\Lamref_A$ as $\Lamref_A$ refines $\Lam_A$ and $\Lamref_A$ is contained in $ A^0$.
%%% (so if $\calR$ meets $g$ in $A$ then $g\subset \calR$).

To see (\ref{refdecA-it- not meeting 0-cc}), 
observe that if $\calR$ is a complementary region of $\calE_A$ such that $\Lamref_A\cap\calR\neq\varnothing$, 
then there is a geodesic $g$ in $\Lam_A$ contained in $\calR$, since $\Lamref_A$ is a lamination refining $\calE_A$.
But then on one side of $g$ there must be a geodesic in $A$ that meets $\calR$.

% if  $\calR$ is a complementary region of $\calE_A$ meeting no geodesic of $A$, 
%then every geodesic $g$ contained in $\calR$ is in $A^0$. Since every geodesic in $\calR$ is transverse to some other geodesic in $\calR$,  
%this implies that  $\calR$ cannot contain any geodesic in $A^{00}$. In particular $\calR$ contains no geodesic of $\Lam_A$ since $\Lam_A\subset A^{00}$.

Assertion (\ref{refdecA-it- 0-closed}) follows from  (\ref{refdecA-it- not meeting  0-cc})  as, by Proposition~\ref{prop:decA},  
any closed geodesic $c$ with $i(c,A)=0$ is either in $\calE_A$ or contained in a complementary region $\calR$ of $\calE_A$ meeting no geodesic of $A$. 
\end{proof}

\subsection{Proof of  Theorem~\ref{thm:1}}\label{s:3.3}

Let $\mu$ be a geodesic current on $\Sigma$, $\supp \mu  \subset \calG(\H)$ its support and $A: = \pi(\supp \mu)$. 
Observe that by definition a geodesic $g$ in $\Sigma$ is $\mu$-short if some (and hence every) lift $\wt{g}$ of $g$ satisfies $i(\wt{g}, \supp \mu) = 0$. 
This implies that $\calE_A = \calE_\mu$ and $\Lambda_A = A^0 \cap A^{00} = \Lambda_\mu$. 
Then Theorem~\ref{thm:1}(1) follows from Proposition~\ref{prop:decA} and Theorem~\ref{thm:1}(2) follows from Proposition~\ref{prop:refdecA}.

\section{Straight pseudo-distance, length shortening and systole}\label{sec:4}

The main objective of this section is to show
 that the systole of a current on a finite
area surface $\Sigma$ with geodesic boundary can be computed using
simple closed geodesics, provided $\mathring{\Sigma}$ is not the
thrice punctured sphere (Corollary~\ref{cor:SimpleSyst}).
This relies on two main ingredients:
\be
\item the fact that a geodesic current gives rise to an appropriate pseudo-distance on $\H$ and 
that the length function associated to this pseudo-distance behaves very much like the hyperbolic length. 
A similar study, in the case of closed surfaces, was carried out by Glorieux in \cite{Glorieux}.
\item A Length-Shortening-Under-Surgery Lemma in the spirit of \cite{Martone_Zhang},
with the additional difficulty due to the presence of boundary components.
\ee

\medskip
A pseudo-distance $d$ on $\H$ is a symmetric function $d: \H \times \H \rightarrow [0,\infty)$ vanishing on the diagonal and verifying the triangle inequality. 
We say that a pseudo-distance $d$ is {\em straight} 
if whenever three points $x,y,z$ lie on a geodesic segment on $\H$ (for the hyperbolic metric) in this order, we have
\bqn
d(x,y) + d(y,z) = d(x,z).
\eqn
We emphasize that such a pseudo-distance is not necessarily continuous for the standard topology on $\H$.

We now turn to the construction of a straight pseudo-distance associated to a geodesic current on $\H$.
Given a geodesic current $\mu$ on $\H$, define for $x,y \in \H$:
\bqn
d_\mu(x,y) = \frac{1}{2} \;\big\{\mu(\calG^\pitchfork_{[x,y)}) + \mu(\calG^\pitchfork_{(x,y]})\big\}
\eqn
where for a possibly empty geodesic segment $I \subset \H$ we define
\bqn
\calG^\pitchfork_I = \{g \in \calG(\H): |g \cap I | = 1\}.
\eqn

\medskip
\begin{center}
\begin{figure}[h]
\begin{tikzpicture}
\draw (0,0) circle [radius=2];
\draw (-1.414,1.414) -- (1.414,-1.414);
\draw (-2,0) arc (90: 37:4);
\draw (0,2) arc (180: 233:4);
\draw (-0.8,1.83) arc (201.5:231.6:7.5);
\draw (-1.88,.7) arc (70:40:7.5);
% geodesic arc I
\draw (.5,0.05) arc (125:147.5:2);
\draw (.5,0.1) node[right] {$I$};
\begin{scope}[rotate=30]
\draw (-1.88,.7) arc (70:40:7.5);
\end{scope}
\begin{scope}[rotate=-40]
\draw (-0.8,1.83) arc (201.5:231.6:7.5);

\end{scope}
\end{tikzpicture}
\caption{ Some of the geodesics in $\calG^\pitchfork_I$}
\end{figure}
\end{center}

\begin{Proposition}\label{prop4.1}
The function $d_\mu$ is a straight pseudo-distance.
\end{Proposition}

\begin{proof}
By definition $d_\mu$ is symmetric and vanishes on the diagonal.

In order to check the triangle inequality we may assume that the points $x,y,z$ are pairwise distinct. 

\noindent
\begin{minipage}{.65\textwidth}
If $g\in\calG^\pitchfork_{[x,y)}$, then either $g \in \calG^\pitchfork_{(x,y)}$ or $g \cap [x,y) = \{x\}$.
If $g \in \calG^\pitchfork_{(x,y)}$, then either $g \in \calG^\pitchfork_{[x,z)}$ or $g \in \calG^\pitchfork_{[z,y)}$.  
If on the other hand $g \cap [x,z) = \{x\}$, then either $g \in \calG^\pitchfork_{[x,y)}$ or otherwise $g$ contains the segment $[x,y)$ 
which implies that $x,y,z$ don't lie on a geodesic and hence $g \in \calG^\pitchfork_{[y,z)}$. 
\end{minipage}
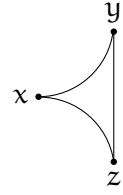
\begin{minipage}{.35\textwidth}
\begin{center}
\begin{tikzpicture}
\filldraw (-.5,0) circle [radius=1pt] node[left] {$x$};
\filldraw (.5,.866) circle [radius=1pt] node[above] {$y$};
\filldraw (.5,-.866) circle [radius=1pt] node[below] {$z$};
%\draw (.5,-.866) arc (300:420:1);
\draw (-.5,0) arc (270:354:1);
\draw (.5,-.866) arc (8:90:1);
\draw (.5,.866) to (.5,-.866);
\end{tikzpicture}
\end{center}
\end{minipage}
This shows the inclusion
\bqn
\calG^\pitchfork_{[x,z)} \subset \calG^\pitchfork_{[x,y)} \cup \calG^\pitchfork_{[y,z)} .
\eqn
An analogous argument shows the corresponding statement for $(x,z]$ and concludes the proof of the triangle inequality. 

If the three points $x,y,z$ lie on a geodesic in this order, 
then $\calG^\pitchfork_{[x,z)}$ is the disjoint union of $\calG^\pitchfork_{[x,y)}$ and $\calG^\pitchfork_{[y,z)}$ and analogously for $(x,z]$. 
This implies that $d_\mu(x,z) = d_\mu(x,y) + d_\mu(y,z)$ and hence $d_\mu$ is straight.

\end{proof}

\begin{Example}\label{exam4.2} ~

\begin{enumerate}
\item
  If $\mu =\calL$  is the Liouville current on $\H$, the corresponding pseudo-distance $d_\mu$ is the hyperbolic metric $d_{\rm hyp}$.

\medskip
\item
  If $\mu \in \calC(\H)$ is a geodesic current on $\H$ such that $\supp(\mu) \subset \calG(\H)$ is a geodesic lamination, 
  a standard argument shows that the quotient metric space $X_\mu = \H/\!\!\sim$, 
  obtained by identifying points at $d_\mu$-distance zero, is $0$-hyperbolic in the sense of Gromov and 
  can therefore be canonically embedded in a complete $\R$-tree (see \cite{MS91} for instance).
\end{enumerate}
\end{Example}

Given a straight pseudo-distance $d$ on $\H$ we define as usual the length $L(c)$ of a continuous path $c:[r,s] \rightarrow \H$, $-\infty < r \le s < + \infty$ by
\bqn
L(c) = \sup\left\{ \displaystyle\sum\limits^n_{i=0} d\big(c(t_i), c(t_{i+1})\big)\!: \, n \ge 1, \,t_0 = r \le t_1 \le \dots \le t_n = s\right\}.
\eqn

\noindent
Of course $L(c)$ is invariant by monotone continuous reparametrization. The statements in the following lemma are straightforward verifications.

\begin{Lemma}\label{lem4.3} The length function $L$ associated to a straight pseudo-distance $d$ has the following properties:
\begin{enumerate}
\item
 If $c:[r,s]\to\H$ is a continuous path, then $L(c) \ge d(c(r),c(s))$ and if $c$ parametrizes a geodesic segment, equality holds.
\item
  If a path $c$ is the concatenation $c = c_1 * c_2$ of two paths $c_1,c_2$, then 
\bqn
L(c) = L(c_1) + L(c_2).
\eqn
\end{enumerate}
\end{Lemma}

Let now $\Gamma < \PSL(2,\R)$ be a torsion-free discrete subgroup and $\Sf : = \Gamma \backslash \H$ be the corresponding quotient surface.
Given a $\Gamma$-invariant straight pseudo-distance $d$ on $\H$ 
we define the length $L(c)$ of a continuous path $c: [r,s] \rightarrow \Sf$ as the length $L(\wt{c})$ of any continuous lift $\wt{c}: [r,s] \rightarrow \H$.

\medskip
The following generalizes a fundamental property of hyperbolic length to length functions associated to straight pseudo-distances.

\begin{Proposition}\label{prop4.4}
Let $s\subset S$ be a closed geodesic represented by a hyperbolic element $\gamma\in\Gamma$ and 
let $p$ be a point on the axis of $\gamma$.  Then for every closed loop $c$ in the free homotopy class of $s$
%
%Let $c: [0,1] \rightarrow \Sf$, $c(0) = c(1)$, be a closed loop whose homotopy class is represented by a hyperbolic element $\gamma \in \Gamma$. 
%Let $p$ be a point on the axis of $\gamma$ and $s: [0,1] \rightarrow \Sf$ be a parametrization of the closed geodesic 
%obtained by projecting the axis of $\gamma$ to $\Sf$. Then
\bqn
L(c) \ge L(s) = d(p,\gamma p).
\eqn
\end{Proposition}

Using the definitions of the length $L$ of a curve in $\Sf$ and Lemma~\ref{lem4.3}, 
the above proposition is a direct consequence of the following:

\begin{Lemma}\label{lem4.5}
With the hypotheses of Proposition~\ref{prop4.4} 
\bqn
d(q,\gamma q) \ge d(p, \gamma p) \;\text{ for all }q \in \H\,.
\eqn
\end{Lemma}

\begin{proof}
Using that $d$ is straight, $\Gamma$-invariant, and applying the triangle inequality we obtain for all $n \ge 1$:
\begin{align*}
n d(p,\gamma p) & = \displaystyle\sum\limits^{n-1}_{i=0} d(\gamma^i p,\gamma^{i+1} p)
\\[1ex]
&= \;d(p,\gamma^n p)
\\ 
&\le  \;d(p,q) + \displaystyle\sum\limits^{n-1}_{i=0}  d(\gamma^i q, \gamma^{i+1} q) + d(\gamma^n q, \gamma^n p)
\\[1ex]
&= \; 2d(p,q) + n \,d(q,\gamma q) .
\end{align*}

\noindent
Dividing by $n$ and letting $n$ tend to infinity we obtain the lemma. 
\end{proof}

Assume now that $\Gamma$ is finitely generated.
Let $C \subset \H$ be the closed convex hull of the limit set of $\Gamma$ and $\Sigma =\Gamma \backslash C$ be the quotient surface, 
which is a complete hyperbolic surface with geodesic boundary and finite area, included in $\Sf = \Gamma \backslash \H$.
Let $\mathring\Sigma$ be the interior of $\Sigma$. Define
\begin{align*}
\Syst(L)& : = \inf\Big\{L(c):\, c \subset \mathring\Sigma\; \text{ is a closed geodesic}\Big\}
\\
\Syst_s(L)& : = \inf\Big\{L(c):\, c \subset \mathring\Sigma \; \text{ is a simple closed geodesic}\Big\}.
\end{align*}

Our objective is to show:

\begin{Proposition}\label{prop4.6}
Let $\Sigma = \Gamma \backslash C$ be a finite area surface with geodesic boundary and 
$L$ be the length function associated to a $\Gamma$-invariant straight pseudo-distance on $\H$.

\begin{enumerate}
\item
  If  $\mathring\Sigma$ is not the thrice punctured sphere, $\Syst(L) = \Syst_s(L)$.

\item
  If $\mathring\Sigma$ is the thrice punctured sphere 
  \bqn
  \Syst(L) = \min\{L(c):\, c \subset \mathring\Sigma\text{ is a closed geodesic with }i(c,c) = 1\}\,.
  \eqn
\end{enumerate}
\end{Proposition}

Before indicating the proof of Proposition~\ref{prop4.6} we establish the link with the systole of a current. 
Let thus $\mu$ be a geodesic current on $\Sigma$, 
and let $d_\mu$ be the invariant straight pseudo-distance on $\H$ and $L_\mu$ be the corresponding length function on curves in $\Sf = \Gamma \backslash \H$.

\begin{Lemma}\label{lem4.7}
Let $c \subset \Sf$ be a closed geodesic. Then $i(\mu,c) = L_\mu(c)$.
\end{Lemma}

\begin{proof}
Let $\gamma \in \Gamma$ be a hyperbolic element representing $c$ and let
$p$ be a point on its axis $g$. Then we have (by Proposition~\ref{prop4.4}):
\bqn
L_\mu(c)  = d_\mu (p,\gamma p)
=\frac{1}{2} \;\left\{\mu(\calG^\pitchfork_{[p,\gamma p)}\big) + \mu \big(\calG^\pitchfork_{(p, \gamma p]}\big)\right\}.
\eqn

Observe that $\calG^\pitchfork_{[p,\gamma p)}\times\{g\}$ as well as
  $\calG^\pitchfork_{(p,\gamma p]}\times\{g\}$ are Borel fundamental
domains for the $\Gamma$-action on $(\calG(\H) \times \supp \delta_c  ) \cap \Dd \calG (\H)$ and hence their $\mu$-measure equals $i(\mu,c)$ by definition.
\end{proof}

Lemma~\ref{lem4.7} and Proposition~\ref{prop4.6} now lead to the main result of this section concerning the systole of $\mu$.

Recall that 
\bqn
\Syst(\mu) = \inf\,\{i(\mu,c):\,c \text{ is a closed geodesic in }\mathring\Sigma\}
\eqn 
and 
\bqn
\Syst_s(\mu) = \inf\, \{i(\mu,c):\, c\text{ is a simple closed geodesic in }\mathring\Sigma\}\,.
\eqn 
Then

\begin{Corollary}\label{cor:SimpleSyst}
Let $\mu$ be a geodesic current on a finite area hyperbolic surface $\Sigma$ with geodesic boundary. 

\begin{enumerate}
\item
  If $\mathring\Sigma$ is not the thrice punctured sphere,
\bqn
\Syst(\mu) = \Syst_s(\mu).
\eqn

\item
  If $\mathring\Sigma$ is the thrice punctured sphere and the carrier of $\mu$ is not included in $\partial \Sigma$,
\bqn
\Syst(\mu) > 0.
\eqn
\end{enumerate}
\end{Corollary}

\begin{proof}
(1) This statement follows from Proposition~\ref{prop4.6}(1) and Lemma~\ref{lem4.7}.

\medskip
\noindent
(2) % If $\mu$ gives positive mass $\lambda$ to some closed geodesic $c_0$ of self-intersection $1$, then we have $i(\mu, c_0) \ge \lambda \,i(c_0,c_0) = \lambda > 0$ and for any other closed geodesic $c'$, $i(\mu,c') \ge \lambda i(c_0,c') \ge \lambda > 0$ which shows $\Syst(\mu) \ge \lambda > 0$.
%
%\smallskip
%If $\mu$ gives zero mass to all closed geodesics of self-intersection $1$, 
Since the carrier of $\mu$  is not contained in the boundary of $\Sigma$,  
there is a geodesic $g \subset \mathring\Sigma$ in the carrier of $\mu$. But $g$ intersects then transversally at least one closed geodesic with one self-intersection. 
The claim follows then from Proposition~\ref{prop4.6}(2).% which is not closed of self-intersection $1$; Since $\mathring\Sigma$ is a thrice punctured sphere, $g$ intersects transversally every closed geodesic of self-intersection $1$ and hence $\Syst(\mu) = \min\{i (\mu,c)$: $c \subset \mathring\Sigma, i (c,c) = 1\} > 0$.
\end{proof}

\medskip
We now indicate the main steps in the proof of Proposition~\ref{prop4.6}. 
It relies on a ``Length-Shortening-Under-Surgery'' property and standard arguments from surface topology. 
In the sequel we will use the notation $i(c,c')$ for distinct loops $c,c'$ 
as the minimum intersection number of loops in the free homotopy classes represented by $c$ and $c'$. {Instead we will denote by $i(c,c)$ the number of self-intersections of $c$.}

Let now $c \subset \Sf$ be a closed geodesic with at least one self-intersection point $p \in c$. Choose a parametrisation
$c: [0,1] \rightarrow \Sf$ and $t \in (0,1)$ such that $p = c(0) =
c(1) = c(t)$. Then $c$ is the concatenation of $c_1 = c|_{[0,t]}$ and
$c_2 = c|_{[t,1]}$; let also $\ov{c}_2$ denote the loop $c_2$
with opposite orientation and let $c_3 = c_1 * \ov{c}_2$ be the
concatenation of $c_1$ and $\ov{c}_2$.

\medskip
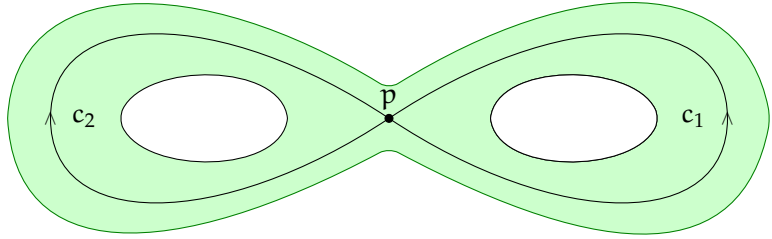
\begin{center}
%\begin{figure}[h]\label{fig:1}
\begin{tikzpicture}[scale=1.5]
%outer curve 
\filldraw[fill=green!20!white, draw=green!50!black, rounded corners, scale=1.3] (-2.6,0) .. controls (-2.55,1) and (-1.3,1) .. (0,0.2) .. controls (1.3,1) and (2.4,1) .. (2.6,0) .. controls (2.4,-1) and (1.5,-1) .. (0,-0.2) .. controls (-1.5,-1) and (-2.55,-1) .. (-2.6,0);
%middle curve
\draw (0,0) .. controls (1.5,1) and (3,1) .. (3,0) .. controls (3,-1) and (1.5,-1) .. (0,0) .. controls (-1.5,-1) and (-3,-1) .. (-3,0) .. controls (-3,1) and (-1.5,1) .. (0,0);
%right middle arrow
\draw (3,0) node [rotate=270] {\tiny{$<$}};
%left middle arrow
\draw (-3,0) node [rotate=270] {\tiny{$<$}};
%%p
\filldraw (0,0) circle [radius=1pt];
%gamma_2
\draw (2.7,0) node {\small{$c_1$}};
%gamma_3
\draw (-2.7,0) node {\small{$c_2$}};
%
%%%right inner curve
\filldraw[fill=white, draw=black, rounded corners, scale=.5, xshift=1.8cm] (0,0) .. controls (.1,1) and (2.5,1) .. (3,0) .. controls (2.5,-1) and (.1,-1) .. (0,0);
\filldraw[fill=white, draw=black, rounded corners, scale=.5, xshift=1.8cm] (0,0) .. controls (.1,1) and (2.5,1) .. (3,0) .. controls (2.5,-1) and (.1,-1) .. (0,0);
%%%left inner curve
\filldraw[fill=white, draw=black, rounded corners, scale=.5, xshift=-1.8cm] (0,0) .. controls (-.1,1) and (-2.5,1) .. (-3,0) .. controls (-2.5,-1) and (-.1,-1) .. (0,0);
\node at (0,0) [above] {$p$};
\end{tikzpicture}
%\caption{The three curves obtained from $c$ resolving the self intersection at $p$.}
%\end{figure}
\end{center}

Observe that 
\bqn
L(c_3) = L(c_1) + L(c_2) = L(c)  \;\mbox{and} \; i(c_i,c_i) < i(c,c), \, i \in \{1,2,3\}.
\eqn

The following is an immediate consequence of  Proposition~\ref{prop4.4}:
\begin{Lemma}\label{lem4.9}
If for $i \in \{1,2,3\}$, if $c_i$ is freely homotopic to the closed geodesic $s_i$, then $L(s_i) \le L(c)$.
\end{Lemma}

Given a closed geodesic $c \subset \Sigma$ with positive self-intersection, let $\Sigma_c \subset \Sigma$ denote the subsurface filled by $c$,
that, we recall, is obtained by taking a regular tubular neighborhood of $c$ and 
adding to it all the components of the complement that are either simply connected or whose fundamental group is cyclic. % generated by a parabolic element. 
 Standard arguments in surface topology, combined with Lemma~\ref{lem4.9} then imply:

\begin{Lemma}\label{lem4.10} ~

\begin{enumerate}
\item
  If $c \subset \Sigma$ is a closed self-intersecting geodesic, we have for every connected component $c'$ of $\partial \Sigma_c$:
\bqn
L(c') \le L(c).
\eqn

\item
  Given a closed geodesic $c \subset \Sigma$ with positive self-intersection, 
  there exists a closed geodesic $c'$ in $\Sigma_c$ with $i(c', c') = 1$ and $L(c') \le L(c)$. In particular  
$\mathring\Sigma_{c'}$ is a thrice punctured sphere.
\end{enumerate}
\end{Lemma}

\begin{proof}[Proof of Proposition~\ref{prop4.6}] ~

\begin{enumerate}
\item
  Let $c \subset \mathring\Sigma$ be a closed geodesic with positive self-intersection; 
  by Lemma~\ref{lem4.10}(2) there is a closed geodesic $c'$ in $\mathring\Sigma$ with $i(c',c') = 1$ and $L(c') \le L(c)$. 
  Since $\mathring\Sigma_{c'}$ is a thrice punctured sphere and $\mathring\Sigma$ is not, there is a boundary component $c''$ of 
$\Sigma_{c'}$ which is a simple closed geodesic contained in $\mathring\Sigma$; Lemma~\ref{lem4.10}(1) 
implies then that $L(c'') \le L(c') \le L(c)$ and this shows that $\Syst(L) = \Syst_s(L)$.

\medskip
\item
  Follows immediately from Lemma~\ref{lem4.10}(2).
\end{enumerate}
\end{proof}

\begin{Remark}
Lemma~\ref{lem4.9} is Proposition~4.5 in \cite{Martone_Zhang} for the case where $S = \Gamma \backslash \H$ is compact. 
We believe that the use of pseudo-distance associated to a current simplifies the arguments.
\end{Remark}

\section{Currents with positive systole}\label{sec:5}

This section is devoted to the proofs of Theorem~\ref{thm_intro:positive systole} 
and Corollary~\ref{cor_intro:MCG}.

\medskip
Let then $\Gamma < \PSL (2, \R)$ be finitely generated and torsion-free, $\Sf = \Gamma \backslash \H$ the quotient surface and 
$\Sigma = \Gamma \backslash C$ the corresponding finite area surface with geodesic boundary.

Let $\mu \in \calC(\Sigma)$ be a geodesic current on $\Sigma$ and 
recall that a geodesic $c \subset \Sigma$ is $\mu$-short 
if some (and hence any) lift $\wt{c}$ of $c$ does not intersect transversally any geodesic in the support of $\mu$. 
The main ingredients in the proof of Theorem~\ref{thm_intro:positive systole} are the results on systoles established in \S~\ref{sec:4} 
together with the following proposition, establishing the implication (1) $\Longrightarrow$ (4) of Theorem~\ref{thm_intro:positive systole}.

\begin{Proposition}\label{prop5.1}
Let $\mu$ be a geodesic current on $\Sigma$ and $c: \R \rightarrow
\Sigma$ a $\mu$-short geodesic that is recurrent in
$\mathring\Sigma$. Then
for all $\eps>0$, there exists a closed geodesic $c_\eps$
in $\mathring\Sigma$ such that $i(\mu,c_\eps)<\eps$.
%% there exists a sequence $(c_n)_{n \ge 1}$ of closed geodesics in $\mathring\Sigma$ such that $\lim_{n \rightarrow \infty} \,i(\mu,c_n) = 0$.
\end{Proposition}

The hyperbolic metric on $\H$ and $\Sf$ induces Riemannian metrics on
the respective unit tangent bundles $T^1\H$ and $T^1 S$, denoted
$d^*$, for which the projection maps are Riemannian submersions; as
usual $d_{\rm hyp}$ will denote the hyperbolic distance on $S$ and
$\H$. We denote by $g_t$ the geodesic flow action on $T^1 \H$.

We will use:

%% \medskip\noindent
%% {\bf Closing Lemma}
\begin{ClosingLemma}{\cite[4.5.15]{Eb96}}
    Given a compact set $C \subset T^1(\H)$ and $\zeta > 0$, there
  exist $T \ge 0$ and $\delta > 0$ such that if there is $t \ge T$, $v
  \in C$ and $\gamma \in \PSL(2,\R)$ with $d^* (\gamma(v),g_t(v)) <
  \delta$, then there is
  $t'\in\R$ with $|t' - t| < \zeta$
  and $v' \in T^1 \H$ with $d^*(v',v) < \zeta$
  and $\gamma(v') = g_{t'}(v')$. 
\end{ClosingLemma}

\begin{proof}[Proof of Proposition~\ref{prop5.1}]
     We may suppose that $c$ is not closed (otherwise the statement is clear).
Recall that $c$ is recurrent in $\mathring\Sigma$ if there exists a
sequence $(t_n)_{n \ge 1}$ in $\R$ with $\lim_n |t_n| = \infty$ and
$\{c(t_n): n \ge 1\}$ stays in a fixed  compact subset $K$ of
$\mathring\Sigma$. Modulo reparametrizing with opposite orientation,
we may assume that the sequence $(t_n)_{n \ge 1}$ is monotone increasing with $\lim
t_n = + \infty$. 
Let $v \in T^1_p \,\mathring\Sigma$ be an accumulation point of the
sequence $(\dot{c}(t_n))_{n \ge 1}$; enlarging $K$ we may assume $p
\in \mathring{K}$. Let $s: \R \rightarrow \Sf$ be the unit speed
geodesic with $\dot{s}(0) = v$, and $g \subset \H$ be a lift of $s$.
\begin{claim}
	For all $\epsilon>0$, and for all $\wt p\in g$ except at most countably many, there exists $\eta>0$ so that $\mu(\calG_{\ov{B}(\wt{p},\eta)}
	\backslash \{g\}) < \eps$.
\end{claim}	
\begin{proof}\phantom\qedhere
Let $t_0 > 0$ be such that $s((-t_0,t_0))
\subset \mathring{K}$. Then for any $t \in (-t_0,t_0)$, $\dot{s}(t)$
is an accumulation point of the sequence $(\dot{c}(t_n + t))_{n \ge 1}$.
Let $g \subset \H$ be a lift of $s$; we claim that the set of points
$x \in g$ such that $\mu(\calG_{\{x\}}) > \mu (\{g\})$ is at most
countable. Indeed, the family of Borel subsets $\{{\calG}_{\{x\}}
\backslash \{g\}\}_{x \in g}$ are pairwise disjoint; since $\mu$ is
$\sigma$-finite, the claim follows. Thus replacing $p = s(0)$ by
$s(t)$ for some appropriate $t \in (-t_0,t_0)$ we may assume that $p$
is the projection of a point $\wt{p} \in g$ such that $\mu
(\calG_{\{\wt{p}\}} \backslash \{g\}) = 0$.
%% Possibly decreasing $\eps$
%% we may suppose that $\eps < d_{{\rm hyp}}(K, \partial \Sigma)$.
Since $\mu (\calG_{\{\wt{p}\}} \backslash \{g\}) = 0$, we may now
choose $\eta > 0$ such that $\mu(\calG_{\ov{B}(\wt{p},\eta)}
\backslash \{g\}) < \eps$, where $\ov{B}(\wt{p},\eta) \subset \H$ is
the closed ball centered at $\wt{p}$ of radius $\eta$ for the
hyperbolic metric.
\end{proof}
Possibly decreasing $\eta$ we may in addition assume that $\eta$ is
smaller than $d_{{\rm hyp}}(p, \partial \Sigma)$ and the injectivity
radius at $p$. In particular, the projection $\pi: \H \rightarrow \Sf$
sends $\ov{B}(\wt{p},\eta)$ isometrically to $\ov{B}(p,\eta)$, the
corresponding metric ball in $\Sf$, and $\ov{B}(p,\eta)
\subset\mathring\Sigma$.
% $\ov{B}(p,\eta)\subset \mathring \Sigma$
%
%%% Same side of $s$
As $c$ is not closed, $\dot{c}(t_n)$ is never tangent to $s$, so we
may assume that $c(t_n)$ never belongs to $s$. Passing to a
subsequence we may now assume that all points $c(t_n)$ are on the same side
of $s$ in $\ov{B}(p,\eta)$.
%% namely in the same connected component of
%% $\ov{B}(p,\eta) \backslash \{s(t): |t| \le \eta\}$.

\begin{claim}
If $\alpha:[0,t_n - t_m + 1] \rightarrow\mathring\Sigma$ is a closed loop obtained by concatenation of $c|_{[t_m,t_n]}$ and the
geodesic segment joining $c(t_n)$ and $c(t_m)$ in
$\ov{B}(p,\eta)$, then
$$L_\mu({\alpha})\leq \eps$$ where $L_\mu$ is the length function corresponding to $\mu$.
\end{claim}
\begin{proof}\phantom\qedhere
We have
\bqn
L_\mu({\alpha}) = L_\mu\big({\alpha}|_{[0,t_n - t_m]}\big) + L_\mu \big({\alpha}|_{[t_n - t_m, t_n - t_m + 1]}\big).
\eqn
The first summand vanishes because $c$ is $\mu$-short; for the second
summand observe that since $c(t_n)$ and $c(t_m)$ are on the same side
of $s$ in $\ov{B}(p,\eta)$, the geodesic segment $[{\alpha}(t_n -
  t_m), {\alpha}(t_n - t_m + 1)]$ is disjoint from $c$, contained
in $\ov{B}({p}, \eta)$ and as a result
\bqn
L_\mu\big({\alpha}\big|_{[t_n - t_m, t_n - t_m + 1]}\big) \le \mu \big(\calG_{\ov{B}(\wt{p},\eta)} \backslash \{g\}\big) < \eps.
\eqn
\end{proof}
%%% apply Closing Lemma
Now let $0 < \zeta < \frac{\eta}{2}$ and let  $C$ be the compact set
consisting of unit tangent vectors based at a point of $\ov{B}
(\wt{p}, \frac{\eta}{2})$; let $T$ and $\delta$ be the corresponding
constants given by the Closing Lemma. We may assume $\delta <
\frac{\eta}{2}$ and choose $n_0 \in \N$ such that $d^*(\dot{c}(t_n),v)
< \frac{\delta}{2}\; \forall n \ge n_0$. We can pick $n > m \ge n_0$
so that $t_n - t_m \ge T$ and we have
\bq\label{eq:ClosingLemma}
d^*\big(\dot{c}(t_n), \dot{c}(t_m)\big) < \delta\,.
\eq
Let $\alpha: [0,t_n - t_m + 1] \rightarrow \mathring\Sigma$ be the
closed loop obtained by concatenation of $c|_{[t_m,t_n]}$ and the
geodesic segment joining $c(t_n)$ and $c(t_m)$ in
$\ov{B}(p,\eta)$. Let $\wt{\alpha}: [0,t_n - t_m + 1] \rightarrow \H$
be the unique lift with $d(\wt{\alpha}(0),\wt{p}) < \eta$ and let
$\gamma \in \Gamma$ be such that $\gamma \,\wt{\alpha}(0) =
\wt{\alpha}(t_n - t_m + 1)$.  Then it follows from
(\ref{eq:ClosingLemma}) that
\bqn
d^*\big(\gamma \dot{\wt{\alpha}}(0), \;\dot{\wt{\alpha}}(t_n - t_m) \big)< \delta
\eqn
and hence it follows from the Closing Lemma that
%% $v'$ with $d^*(v',\dot{\wt{\alpha}}(0)) < \nu < \frac{\eta}{2}$ and $t'$ such that $|t' - (t_n - t_m)| < \nu$ with $\gamma \,v' = g_{t'}(v')$.
$\gamma$ is hyperbolic and that its axis contains a point $\wt{p}'$
with $d_{\rm hyp}(\wt{\alpha}(0),\wt{p}') < \zeta $, in particular
$\wt{p}'\in \ov{B}(\wt{p},\eta)$.  The projection to $S$ of the axis
of $\gamma$ gives us then a closed geodesic $c_\eps$ contained in
$\mathring\Sigma$ and for which $i(\mu, c_\eps) \le L_\mu (\alpha) <\epsilon$ (see Proposition~\ref{prop4.4} and Lemma~\ref{lem4.7}).
%where $L_\mu$ is the length function corresponding to $\mu$.
%
%Finally we estimate $L_\mu(\wt{\alpha})$: we have
%\bqn
%L_\mu(\wt{\alpha}) = L_\mu\big(\wt{\alpha}|_{[0,t_n - t_m]}\big) + L_\mu \big(\wt{\alpha}_{[t_n - t_m, t_n - t_m + 1]}\big).
%\eqn
%The first summand vanishes because $c$ is $\mu$-short; for the second
%summand observe that since $c(t_n)$ and $c(t_m)$ are on the same side
%of $s$ in $\ov{B}(p,\eta)$, the geodesic segment $[\wt{\alpha}(t_n -
 % t_m), \wt{\alpha}(t_n - t_m + 1)]$ is disjoint from $g$, contained
%in $\ov{B}(\wt{p}, \eta)$ and as a result
%\bqn
%L_\mu\big(\wt{\alpha}\big|_{[t_n - t_m, t_n - t_m + 1]}\big) \le \mu \big(\calG_{\ov{B}(\wt{p},\eta)} \backslash \{g\}\big) < \eps
%\eqn
%which shows that $i(\mu,c_\eps) < \eps$ and 
This concludes the proof.
\end{proof}

\begin{proof}[Proof of Theorem~\ref{thm_intro:positive systole}]
  We will show that the contrapositions of properties (1), (2), (3), (4), denoted (1)', (2)', (3)', (4)', are equivalent.

\medskip\noindent
(1)' $\Longrightarrow$ (2)': Assume $\Syst(\mu) = 0$. Since $\mathring\Sigma$ is not the thrice punctured sphere,
we have $\Syst_s(\mu) = 0$ by Corollary~\ref{cor:SimpleSyst}(1). 
Let thus $(c_n)_{n \ge 1}$ be a sequence of simple closed geodesics in $\mathring\Sigma$ 
with $\lim_{n \rightarrow \infty} \,i(\mu, c_n) = 0$, in particular 
\bqn
\lim_{n \rightarrow \infty} \,i(\mu,\frac{c_n}{\ell_{\mathrm {hyp}}(c_n)}) = 0\,.
\eqn 
Now the sequence of currents $(\frac{\delta_{c_n}}{\ell_{\mathrm {hyp}}(c_n)})_{n \ge 1}$ is contained in $\{\nu \in \MLc(\mathring\Sigma): i (\calL,\nu) = 1\}$. 
Since the latter space is compact (see Proposition~\ref{prop:proper} and the remark preceding Theorem~\ref{thm_intro:positive systole}), 
this sequence has a accumulation point, say $\nu_0$, 
in $\MLc(\mathring\Sigma)$ for which $i(\mu, \nu_0) = 0$ and $i(\calL, \nu_0) = 1$, 
in particular $\nu_0 \not= 0$. This shows the announced implication.

\medskip\noindent
(2)' $\Longrightarrow$ (3)': clear.

\medskip\noindent
(3)' $\Longrightarrow$ (4)': let $\nu \in \calC_K(\Sigma)$, with $i(\mu,\nu) = 0$ and $\nu \not= 0$. 
Then any geodesic $g \in \supp \nu$ does not intersect transversally any geodesic of $\supp \mu$; 
such a geodesic $g$ is $\mu$-short by definition and recurrent since $\pi(g) \subset K \subset \mathring\Sigma$.

\medskip\noindent
(4)' $\Longrightarrow$ (1)': This is the content of Proposition~\ref{prop5.1}.
\end{proof}

\begin{proof}[Proof of Corollary~\ref{cor_intro:MCG}] 
  (1) If $\mathring{\Sigma}$ is the thrice punctured sphere the assertion follows from Proposition~\ref{prop4.6}(2).
  We may hence assume that $\mathring{\Sigma}$ is not the thrice punctured sphere.
  Let $(\mu_n)_{n \ge 1}$ be a convergent sequence in $\calC(\Sigma)$ with limit $\mu$.
  Since $\lim_{n\to\infty} i(\mu_n,c)=i(\mu,c)$ for every closed geodesic $c$,
then $\overline\lim_{n\to\infty}\Syst(\mu_n)\leq\Syst(\mu)$ and hence $\Syst$ is continuous if $\Syst(\mu)=0$.

Let $\Syst(\mu)>0$ and assume by contradiction that 
\bqn
\underline\lim_n\Syst(\mu_n)<\Syst(\mu)\,.
\eqn
For every $n\geq1$, it follows from Proposition~\ref{prop4.6}(1) that 
there exists a simple closed geodesic $c_n$ with 
\bqn
i(\mu_n,c_n)\leq\Syst(\mu_n)+\frac{1}{n}\,.
\eqn
If $\{\ell_\mathrm{hyp}(c_n):\,n\geq1\}$ is unbounded,
without loss of generality we may assume that $\lim_{n\to\infty}\ell_\mathrm{hyp}(c_n)=\infty$
and that the sequence $\delta_{c_n}/\ell_\mathrm{hyp}(c_n)$ 
converges to a compactly supported measured lamination $\nu\in\MLc(\mathring\Sigma)$.
But then 
\bqn
i(\mu,\nu)=\lim_{n\to\infty}\frac{i(\mu_n,\nu_n)}{\ell_\mathrm{hyp}(c_n)}=0\,,
\eqn
which, by Theorem~\ref{thm_intro:positive systole}, implies that $\Syst(\mu)=0$, a contradiction.

Hence, by passing to a subsequence, we may assume that $c_n=c$ for all $n\geq1$,
and thus
\bqn
\Syst(\mu)\leq i(\mu,c)=\lim_{n\to\infty} i(\mu_n,c_n)\leq\underline\lim_n\Syst(\mu_n)\,,
\eqn
which is a contradiction.

\medskip
\noindent
(2) Assume that the first inequality does not hold.  Then there is a sequence of closed geodesics 
$(c_n)_{n\geq}$ contained in $K$ such that 
\bqn
\lim_{n\to\infty}\frac{i(\mu,c_n)}{\ell_\mathrm{hyp}(c_n)}=0\,.
\eqn
Using that 
\bqn
\left\{\frac{\delta_{c_n}}{\ell_\mathrm{hyp}(c_n)}:\,n\geq1\right\}
\eqn
is relatively compact (see Proposition~\ref{prop:proper}(1)), let $\nu\in C_k(\Sigma)$, for $\nu\neq0$,
be an accumulation point of this sequence.  Then $i(\mu,\nu)=0$, which contradicts Theorem~\ref{thm_intro:positive systole}(3).
An analogous argument leads to the second inequality.
%(2) The fact that $\Omega$ is open follows from (1) and
%{\color{red} the second assertion follows from (3).}\marginpar{AP: I find it a little too short. It would be nice to give some more explanation or a reference, for example Labourie2008 annales ENS ?}
%We show that the set $\{\mu \in \calC(\Sigma): \Syst(\mu) = 0\}$ is closed. 
%Let $(\mu_n)_{n \ge 1}$ be a convergent sequence with limit $\mu$ and $\Syst(\mu_n) = 0, \,\forall n \ge 1$. 
%By Theorem~\ref{thm_intro:positive systole}(2) 
%there exists for every $n \ge 1$ a measured lamination $\lambda_n \in \ML_0 (\mathring\Sigma)$ with $\lambda_n \not= 0$ and $i(\mu_n, \lambda_n) = 0$. 
%We may assume that $i(\calL,\lambda_n) = 1$ and take an accumulation point $\lambda \in \MLc (\mathring\Sigma)$ of $(\lambda_n)_{n \ge 1}$. 
%Then $i(\mu,\lambda) = 0$ by continuity of $i$ and $i(\calL,\lambda) = 1$, in particular $\lambda \not= 0$. 
%%The second assertion in (1) follows from (2).

\medskip
\noindent
%(3) This follows from an analogous compactness argument as in (1), using Proposition~\ref{prop:proper}.
%This is a standard argument using the compactness of $\{\nu \in \calC_K(\Sigma): i(\calL,\nu) = 1\}$. 
(3) Let $\varphi_n:\Sigma\to\Sigma$ be a sequence of homeomorphisms fixing pointwise $\partial\Sigma$ 
and such that $\varphi_n\to\infty$ in the mapping class group of $\Sigma$.  
Since $\Omega$ is locally compact, it suffices to show that if $\tau\in\Omega$ and
$[\mu]\in\P\calC(\Sigma)$ is any accumulation point of $[\varphi_n(\tau)]$,
then $\Syst(\mu)=0$.  Let $c\subset\mathring{\Sigma}$ be a closed geodesic such that \bqn
\lim_{n\to\infty}\ell_\mathrm{hyp}(\{{\varphi_n^{-1}(c)}\})=\infty\,,
\eqn
where $\{{\varphi_n^{-1}(c)}\}$ is the closed geodesic in the free homotopy class of $\varphi_n^{-1}(c)$.
Then it follows from the first inequality in (2) that 
\bqn
\lim i(\varphi_n(\tau),c)=\lim i(\tau,\{{\varphi_n^{-1}(c)}\})=\infty\,.
\eqn
Let $(n_k)_{k\geq1}$ be a subsequence and $\lambda_k>0$ such that 
\bqn
\lim_{k\to\infty}\frac{\varphi_{n_k}(\tau)}{\lambda_k}=\mu\,.
\eqn
In particular we have that 
\bqn
\lim_{k\to\infty}\frac{i(\varphi_{n_k}(\tau),c)}{\lambda_k}=i(\mu,c)\,,
\eqn
which implies that $\lim\lambda_k=\infty$.
using the continuity of the systole map, we get that 
\bqn
\Syst(\mu)=\lim_k\frac{\Syst(\varphi_{n_k}(\tau))}{\lambda_k}=\lim_k\frac{\Syst(\tau)}{\lambda_k}=0\,.
\eqn
\end{proof}

\section{Currents with vanishing systoles and laminations}\label{sec:6}

In this section we establish Theorem~\ref{thm_intro:irred and syst 0} which characterizes geodesic currents with vanishing systole 
that occur as components in the decomposition theorem.

The main tools are Theorem~\ref{thm_intro:positive systole} and the following proposition that is of independent interest.

\begin{Proposition}\label{prop6.1}
Let $\mu \in \calC(\Sigma)$ and $\Lam \subset \Sigma$ be a geodesic lamination without isolated leaves and consisting of $\mu$-short geodesics. 
Then for any closed geodesic $c \subset \Sigma$ bounding a crown of a complementary region of $\Lam$ we have
\bqn
i(\mu,c) = 0.
\eqn
\end{Proposition}

Let $\Lamt$ be the lift $\Lam$ to a $\Gamma$-invariant geodesic lamination of $\H$ and let $\calR$ be a complementary region of $\Lamt$. 
Then $\calR$ is bounded by leaves of $\Lamt$ whose endpoints in $\partial \H$ are the vertices of $\calR$. 
We now make the following crucial observation: let $a,b,c$ be consecutive vertices of $\calR$ ordered such that $(a,b,c)$ is positively oriented; 
since $\Lam$ has no isolated leaf, the pencil $\{b\} \times I_{[a,c]}$ does not contain the axis of a hyperbolic element; 
in addition, $b$ is in the limit set of $\Gamma$ and cannot be a cusp since otherwise $\Lam$ would have an isolated leaf. 
Therefore, the hypothesis of Lemma~\ref{lem:pencils} are fulfilled and hence
\bqn
\mu(\{b\} \times I_{[a,c]}) = 0.
\eqn

\begin{Lemma}\label{lem6.2}
Let $x_0,\dots,x_n$ be a sequence of consecutive vertices of a complementary region $\calR$ labelled in such a way that $(x_0,\dots,x_n)$ is positively oriented. 
Then the geodesic $(x_0,x_n)$ is $\mu$-short.
\end{Lemma}

\begin{proof}
The proof proceeds by recurrence. For $n=1$ the statement holds. Let us now suppose $n \ge 2$. We have the following equalities:
\begin{align*}
I_{(x_0,x_n)} & = I_{(x_0,x_{n-1}]} \cup I_{(x_{n - 1},x_n)}
\\
I_{(x_n,x_0)} & = I_{(x_{n-1},x_0)} \cap I_{(x_n,x_{n-2})} \cap I_{(x_n,x_{n - 1})} .
\intertext{Thus}
I_{(x_0,x_n)} & \times I_{(x_n,x_0)} \subset   I_{(x_0, x_{n-1})}\times I_{(x_{n-1},x_0)}
\\
& \cup \{x_{n-1}\} \times I_{(x_{n},x_{n-2})}
\\
&\cup I_{(x_{n-1},x_n)}\times I_{(x_n,x_{n - 1})}.
\end{align*}
Using that $(x_{n-1},x_n)$ is $\mu$-short, the induction hypothesis that $(x_0,x_{n-1})$ is $\mu$-short and 
the observation preceding Lemma~\ref{lem6.2}, we get to the conclusion that $(x_0,x_{n})$ is $\mu$-short.
\end{proof}

\begin{proof}[Proof of Proposition~\ref{prop6.1}] Let $\calC \subset \Sigma$ be a crown in the complement of the lamination $\ov{\Lam}$, 
and let $\gamma \in \Gamma$ be a geodesic bounding $\calC$. 
We choose lifts to $\H^2$ in such a way that the half plane to the left of $(\gamma_+,\gamma_-)$ contains a lift $\wt{\calC}$ of the crown $\calC$.

\medskip
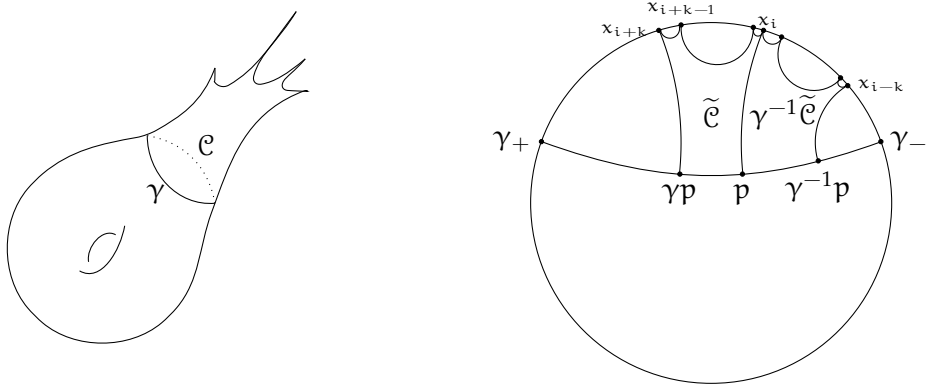
\begin{minipage}{.5\textwidth}
\begin{tikzpicture}[scale=.6]
%surface profile
\draw [rounded corners] 
	(1,2.5) to [out=250, in=30] (-.5,1) 	
	            to [out=190, in=40] (-3,0) 
	            to [out=230, in=130] (-3,-3)  
	            to [out=320, in=220]  (0,-3)
	            to [out=50, in= 250](1,-.5) 
	            to [out=70, in=200] (3,2);
%closed geodesic gamma	        
\draw (-.5,1) to [out=270, in =185] (1,-.5);
\draw[dotted] (1,-.5) to [out=100, in=350] (-.5,1);
%genus
\draw (-2,-2) to [out=330, in= 260] (-1,-1);
\draw (-1.8,-1.8) to [out=100, in=150] (-1.2,-1.2);
%crown
\draw[rounded corners] (1,2.5) to [out=270, in=120] (1.1,2) to [out=30, in=233] (2.8,3.75);
\draw[rounded corners] (2.8,3.75) to [out=239, in=75] (1.8,2.1) to [out=10, in=220] (3, 2.7) ;
\draw[rounded corners] (3, 2.7) to [out=230, in=60] (2.4,2) to [out=360, in=350] (3,2);
%\draw [rounded corners] (1,2.5) to [out=290, in=170] (1.2,2) to [out=30, in=260] (2.8,3.75);
%\draw [rounded corners] (2.8,3.75) to [out=270, in=280]  (1.8,2.1) to [out=20, in=30] (3.8,3.5);
%\draw [rounded corners] (3.8,3.5) -- (2.4,2) -- (3,2);
%\mathcal C
\draw (.8,.8) node {$\mathcal C$};
%gamma
\draw (-.3,-.3) node {$\gamma$};
\end{tikzpicture}
\end{minipage}
\begin{minipage}{.4\textwidth}
\vskip.2cm
\begin{tikzpicture}[scale=.8]
%circle
\draw (0,0) circle [radius=3];
%x
\filldraw (-2.82,1.02) circle [radius=1pt] node[left] {$\gamma_+$};
%y
\filldraw (2.82,1.02) circle [radius=1pt] node[right] {$\gamma_-$};
% geodesic (gamma_-,gamma_+)
\draw (-2.82,1.02) to [out=340, in=200] (2.83,1.02);
%left side of the crown
\draw[xshift=-5.5cm, yshift=1cm, domain=354:382] plot(\x:5);
%right side of the crown
\draw[xshift=5.5cm, yshift=1cm, domain=158:186] plot(\x:5);
%intersezione sinistra
\filldraw (-.87,2.87) circle [radius=1pt] node[left] {\tiny{$x_{i+k}$}};
%intersezione seconda da sinistra
\filldraw (-.5,2.96) circle [radius=1pt] node[above] {\tiny{$x_{i+k-1}$}};;
%intersezione seconda da destra
\filldraw (.7,2.92) circle [radius=1pt];
%intersezione destra
\filldraw (.87,2.87) circle [radius=1pt];
\draw (.95, 3)  node {\tiny{$x_{i}$}};
%left boundary side
\draw [xshift=-.685cm, yshift=2.9cm, domain=190:370] plot(\x:.16);
%middle boudary side
\draw [xshift=.1cm, yshift=2.9cm, domain=175:360] plot(\x:.6);
%right boundary side
\draw [xshift=.778cm, yshift=2.85cm, domain=150:350] plot(\x:.075);
	%\filldraw (1.02,2.83) circle [radius=.5pt];
%R repeated on the right.  In the order: left boudary side, point, middle boundary side, point, right boundary side, point, right side of the crown
\draw [xshift=1.02cm, yshift=2.8cm, domain=163:343] plot(\x:.16);
\filldraw (1.17,2.76) circle [radius=1pt];
\draw [xshift=1.71cm, yshift=2.47cm, domain=150:320] plot(\x:.6);
\filldraw (2.15,2.08) circle [radius=1pt];
\draw [xshift=2.18cm, yshift=2cm, domain=100:350] plot(\x:.075);
\filldraw (2.27,1.95) circle [radius=1pt] node[right] {\tiny{$x_{i-k}$}};
\draw [xshift=2.83cm, yshift=1.01cm, domain=120:198] plot(\x:1.1);
%\gamma^{-1} p
\filldraw (1.78,.7) circle [radius=1pt] node[below] {$\gamma^{-1} p$};
%p
\filldraw (.52,.48) circle [radius=1pt] node[below] {$p$};
%\gamma p
\filldraw (-.52,.48) circle [radius=1pt] node[below] {$\gamma p$};
%\widetilde{\mathcal C}
\draw (0,1.5) node {$\widetilde{\mathcal C}$};
%\gamma^{-1}\widetilde{\mathcal C}
\draw (1.2,1.5) node {$\gamma^{-1}\widetilde{\mathcal C}$};
%x_{i+2}
%\draw 
%x_{i+1}
%x_i
\end{tikzpicture}
\end{minipage}
%
%\psfrag{C}{$\calC$}
%\psfrag{g}{$\gamma$}
%\psfrag{g+}{$\gamma_+$}
%\psfrag{g-}{$\gamma_-$}
%\psfrag{xi+k}{$x_{i+k}$}
%\psfrag{xik1}{$x_{i+k-1}$}
%\psfrag{xi}{$x_i$}
%\psfrag{xi-k}{$x_{i-k}$}
%\psfrag{gp}{$\gamma p$}
%\psfrag{g1p}{$\gamma^{-1} p$}
%\psfrag{p}{$p$}
%\psfrag{Ct}{$\wt{\calC}$}
%\psfrag{g1C}{$\gamma^{-1}\wt{\calC}$}
%\begin{center}
%\includegraphics[width=12.5cm]{fig1.eps}
%\end{center}

\medskip
Then $\wt{\calC}$ has consecutive ideal sides $(x_i,x_{i+1})$, $i \in \Z$, labelled in such a way that $(x_i,x_{i+1},x_{i+2})$ is positively oriented. Now observe that
\bqn
(\gamma_-,\gamma_+) = \lim\limits_{n \rightarrow \infty} \;(x_{-n},x_n).
\eqn
By Lemma~\ref{lem6.2} $(x_{-n},x_n)$ is $\mu$-short, so that
$(\gamma_-,\gamma_+)$ is $\mu$-short since the set of $\mu$-short
geodesics is a closed subset of $\calG(\H)$.
\end{proof}

\begin{proof}[Proof of Theorem~\ref{thm_intro:irred and syst 0}]
Let $\mu_{\mathring\Sigma}$ be a geodesic current as in the statement of Theorem~\ref{thm_intro:irred and syst 0}
and in the remark preceding it. 

\medskip\noindent
(2) $\Longrightarrow$ (1): Follows from Theorem~\ref{thm_intro:positive systole} since $\mu_{\mathring\Sigma} \in \MLc(\mathring\Sigma)$ and $i(\mu,\mu_{\mathring\Sigma}) = 0$.

\medskip\noindent
(1) $\Longrightarrow$ (2): Since $\Syst(\mu) = 0$, 
Theorem~\ref{thm_intro:positive systole} implies the existence of $\nu \in \MLc (\mathring\Sigma)$ with $\nu \not= 0$ and $i(\mu,\nu) = 0$.
Let $\Lamt = \supp \nu$ be the corresponding geodesic lamination and observe that it consists of $\mu$-short geodesics. 
The projection $\Lam$ of $\Lamt$ to $\Sigma$ is a compact subset of $\mathring\Sigma$ by hypothesis. 
An isolated leaf $c$ of $\Lam$ is necessarily a closed geodesic: but $c \subset \mathring\Sigma$ and 
by assumption $i(\mu,c) > 0$, hence $c$ cannot be $\mu$-short. 
Thus $\Lam$ has no isolated leaves. Let $\Lam'$ be a minimal component of $\Lam$; 
then $\Lam'$ satisfies all the assumptions of Proposition~\ref{prop6.1} and since $i(\mu, c) > 0$ for every closed geodesic $c \subset \mathring\Sigma$ 
we deduce that a complementary region of $\Lam'$ in $\Sigma$ is either an ideal polygon, 
an ideal polygon containing one cusp, or an ideal polygon bounding a component of $\partial \Sigma$ \cite[Theorem I.4.2.8]{notesonnotes}. 
We show now how this fact implies that $\supp (\mu_{\mathring\Sigma}) = \wt{\Lam'}$ where $\wt{\Lam'}$ is the lift of $\Lam'$ to $\H$.

Let $g \in \supp (\mu_{\mathring\Sigma})$ and assume that $g$ is not a leaf of $\Lamt '$. 
Since all leaves of $\Lamt '$ are $\mu$-short, $g$ cannot intersect transversally a leaf of $\Lamt '$, 
hence it is contained in a complementary region $\wt{\calR}$ of $\Lamt '$. The specific structure of 
$\wt{\calR}$ implies that if $a,b$ are the endpoints of $g$, one of $a,b$ has to be a vertex of $\wt{\calR}$. 
If $\wt{\calR}$ corresponds to a complementary region of $\Lam'$ bounding a cusp or a crown, 
it has infinitely many vertices and if it is an ideal polygon it must have at least four vertices since $g$ is not a side of $\wt{\calR}$. 
In any case we can find a geodesic $\delta$ connecting two vertices of $\wt{\calR}$ and intersecting $g$ in one point. 
By Lemma~\ref{lem6.2}, $\delta$ is $\mu$-short and this contradicts the assumption that $g \in \supp \mu_{\mathring\Sigma}$. 
Thus $\supp (\mu_{\mathring\Sigma}) \subset \Lamt '$ and by minimality of $\Lam'$ we have equality.
\end{proof}

\section{On the Weyl chamber length  compactification}\label{sec:7}
Let $\Gamma<\PSL(2,\R)$ be a cocompact lattice and $\rho:\Gamma\to G$ a representation.
Recall that when $G=\PSL(n,\R)$, $\rho $ is Hitchin if it lies in the connected component
of $\Hom(\Gamma,\PSL(n,\R))$ containing the restriction to $\Gamma$ of the irreducible $\PSL(n,\R)$-representation of $\PSL(2,\R)$,
while if $G=\Sp(2m,\R)$ it is maximal if the restriction of $\rho$ to some (and hence any) torsion-free subgroup of finite index is maximal
(see \cite{BIW, BIW_htt} for the relevant facts concerning maximal representations).
The space $\calX(\Gamma,G)$ is then the topological space obtained by taking the quotient by $G$-conjugation 
of the set of representations that are Hitchin if $G=\PSL(n,\R)$ or maximal if $G=\Sp(2m,\R)$.

Let $\lambda:G\to\mathfrak{C}$ be the Jordan projection on a closed Weyl chamber $\mathfrak{C}$ and 
let $\calL:\calX(\Gamma,G)\to\P(\mathfrak{C}^\Gamma)$ be defined by $\calL([\rho]):=[\lambda\circ\rho]$,
where $[\rho]$ refers to the $G$-conjugacy class of $\rho$, 
while $[\lambda\circ\rho]$ is the projective class of the length function $\lambda\circ\rho:\Gamma\to\mathfrak{C}$.
The Weyl chamber length boundary of $\calX(\Gamma,G)$ is then defined by
\bq\label{eq:ThP}
\partial\calX(\Gamma,G):=\bigcap_K\overline{\calL(\calX(\Gamma,G)\smallsetminus K)}\,,
\eq
where the intersection is over all compact subsets
$K\subset\calX(\Gamma,G)$ (see \cite{Parreau12}).
For our purposes we make the following choices of Weyl chamber
and describe the corresponding Jordan projection as well as the specific norm we use:
\be
\item If $G=\PSL(n,\R)$, 
\bqn
\mathfrak{C}=\{(x_1,\dots,x_n)\in\R^n:\,x_1\geq\dots\geq x_n\text{ and }\,x_1+\dots+ x_n=0\}
\eqn
and 
\bq\label{eq:wcvtv}
\lambda(g)=(\ln|a_1|,\dotsc,\ln|a_n|),
\eq
where $a_1,\dotsc,a_n$ are the eigenvalues of $g$ counted with multiplicity.
Then for $(x_1,\dots,x_n)\in\mathfrak{C}$, 
\bqn
\|(x_1,\dots,x_n)\|:=x_1-x_n\,.
\eqn
\item If $G=\Sp(2m,\R)$, 
\bqn
\mathfrak{C}=\{(x_1,\dots,x_m)\in\R^m:\,x_1\geq\dots\geq x_m\geq0\}
\eqn
and $\lambda(g)$ is defined as in \eqref{eq:wcvtv}, 
%\bqn
%\nu(g)=(\ln|\lambda_1|,\dotsc,ln|\lambda_n|)
%\eqn
where here however $a_1,\dotsc,a_m$ are the eigenvalues of $g$ of absolute value $\geq1$.
If $(x_1,\dots,x_m)\in\mathfrak{C}$, then
\bqn
\|(x_1,\dots,x_m)\|:=\sum_{i=1}^mx_i.
\eqn
\ee
We will make crucial use of the results in \cite{Martone_Zhang} that establish a relation between length functions and geodesic currents.
In fact, fix $\Gamma_0\vartriangleleft\Gamma$ a torsion-free normal subgroup of finite index, and $S=\Gamma_0\backslash\H$.
Then $\Gamma$ acts on the space $\calC(S)$ of currents on $S$;
the action factors via the finite group $\Gamma/\Gamma_0$ and the space of $\Gamma$-invariant currents $\calC(S)^\Gamma$ is a closed subset of $\calC(S)$.
The following is a direct consequence of \cite{Martone_Zhang}:

\begin{Corollary}\label{cor:7.1}  For every $[\rho]\in\calX(\Gamma,G)$ there is a unique current $\mu_\rho\in\calC(S)^\Gamma$
such that for every $\gamma\in\Gamma_0$
\bqn
i(\mu_\rho,c)=\|\lambda(\rho(\gamma))\|\,,
\eqn
where $c\subset S$ is the closed geodesic corresponding to $\gamma$.
\end{Corollary}

\begin{proof}  Assume $\rho\colon\Gamma\to\PSL(n,\R)$ is Hitchin.  Then $\rho|_{\Gamma_0}\colon\Gamma_0\to\PSL(n,\R)$
is Hitchin as well and there exists a unique $\rho|_{\Gamma_0}$-invariant Frenet curve
$\xi\colon\partial\H\to\calF(\R^n)$ into the variety of full flags
(see \cite[Definition~3.2]{Martone_Zhang} and \cite[Theorem~4.1]{Lab06}).
But then, given any $\eta\in\Gamma$, the assignment $x\mapsto\rho(\eta)^{-1}\xi(\eta x)$
is $\rho|_{\Gamma_0}$-equivariant Frenet as well, and hence coincides with $\xi$.  
The current $\mu$ associated to $\rho|_{\Gamma_0}$ is uniquely determined by its value on rectangles,
that is whenever $(x,y,z,w)$ is a positive $4$-tuple in $\partial\H$, then
\bqn
\mu(I_{[x,y]}\times I_{[z,w]})
=\frac12\{\ln[\xi(x),\xi(y),\xi(z),\xi(w)]+\ln[\xi(z),\xi(w),\xi(x),\xi(y)]\}\,,
\eqn
where $[\,\cdot\,,\,\cdot\,,\,\cdot\,,\,\cdot\,]$ is a specific $\PGL(n,\R)$-invariant of $4$-tuples of complete pairwise
transverse flags (see \cite[Lemma~3.6]{Martone_Zhang}).
This, together with the $\Gamma$-equivariance of $\xi$, implies that $\mu$ is $\Gamma$-invariant.

The argument for maximal representations is completely analogous by using the continuous
$\rho|_{\Gamma_0}$-equivariant map $\xi:\partial\H\to\calL(\R^{2n})$
into the space of Lagrangians that sends positive triples to positive triples.
One concludes by uniqueness that $\xi$ is $\Gamma$-equivariant, 
which implies that the current $\mu$ associated to $\rho|_{\Gamma_0}$
(see \cite[\S~3.2]{Martone_Zhang}) is $\Gamma$-invariant.
\end{proof}
%\begin{proof}  If $\mu$ is the current associated to $\rho|_{\Gamma_0}$ as in \cite{Martone_Zhang},
%then it verifies the identities in Corollary~\ref{cor:7.1} and hence it is uniquely determined.
%The $\Gamma$-invariance  of $\mu$ follows then from the construction of $\mu$ in \cite{Martone_Zhang}
%and the fact that for $G=\SL(n,\R)$ the Frenet curve associated to $\rho|_{\Gamma_0}$ is unique 
%and hence $\rho$-equivariant, while for $G=\Sp(2m,\R)$ a similar uniqueness argument holds.
%\end{proof}

\begin{proof}[Proof of Corollary~\ref{cor_intro:pd}]  Observe that any 
$[L]\in\overline{\calL(\calX(\Gamma,G))}\subset\P(\mathfrak{C}^\Gamma)$
is homogeneous, namely $L(\gamma^k)=kL(\gamma)$ for $k\in\N$.
As a result, $L|_{\Gamma_0}$ does not vanish identically and the map
\bqn
\ba
\calR:\overline{\calL(\calX(\Gamma,G))}&\to\P(\R_{\geq0}^{\Gamma_0})\\
[L]\qquad&\mapsto\,\,[L|_{\Gamma_0}]
\ea
\eqn
is well defined and continuous.

The map $I:\P(\calC(S))\to\P(\R_{\geq0}^{\Gamma_0})$ that to a projectivized current associates
its projectivized intersection function is a homeomorphism onto its image \cite{Otal},
and by \cite[Theorem~1.1, Theorem~3.4 and Corollary~3.11]{Martone_Zhang}, 
its image contains $\calR(\calL(\calX(\Gamma,G)))$ and hence $\calR(\overline{\calL(\calX(\Gamma,G))})$.
Thus 
\bqn
\Omega(\Gamma,G)=\calR^{-1}(I(\Omega(S)))\cap\partial\calX(\Gamma,G)
\eqn
and the assertions of Corollary~\ref{cor_intro:pd} follow from the corresponding ones in Corollary~\ref{cor_intro:MCG}.
\end{proof}

Next we show how Corollary~\ref{cor:split} can be deduced from Theorem~\ref{thm_intro:triangle} 
and \cite[Theorem~B]{ALS18}.  If $\Delta=\Delta(3,3,4)$ and $G=\PSL(n,\R)$ with $n\geq3$
or $G=\Sp(2m,\R)$ with $m\geq3$, then $\calX(\Delta,G)$ is a positive dimensional cell,
in particular $\partial\calX(\Delta,G)\neq\varnothing$.
Since $\Delta$ contains a torsion-free subgroup of index $24$ representing a genus $2$ surface,
any $\Gamma$ as in Corollary~\ref{cor:split} is isomorphic to a torsion-free subgroup of $\Delta$
of finite index and Theorem~\ref{thm_intro:triangle}  implies that $\Omega(\Gamma,G)\neq\varnothing$.

In the case of $\Sp(4,\R)$ one can take $\Delta=\Delta(3,4,4)$; %\marginpar{BP: I don't understand: by definition $\calX(\Delta,G)$ should be the union of all maximal components, but that is disconnected? Do we mean contains? {\color{red} $\calX(\Delta,\mathrm{Sp}(4,\R)$ is non-compact as it contains the $\mathrm{Sp}(4,\R)$ Hitchin component that is homeomorphic to $\R^2$ by \cite[Thm B]{ALS18}}}
then $\calX(\Delta,\mathrm{Sp}(4,\R)$ is non-compact 
as it contains the $\mathrm{Sp}(4,\R)$ Hitchin component that is homeomorphic to $\R^2$ by \cite[Thm B]{ALS18}, 
$\Delta$ contains a torsion-free subgroup of index $12$ representing a genus $2$ surface
and the same argument as above allows us to conclude.
Theorem~\ref{thm_intro:triangle} is in turn an immediate consequence of Corollary~\ref{cor:7.1}
and the following:

\begin{Theorem}\label{thm:7.2}  Let $\Delta<\PSL(2,\R)$ be a hyperbolic triangle group,
$\Gamma<\Delta$ a torsion-free subgroup of finite index and $S=\Gamma\backslash\H$.
Then for any non-vanishing current $\mu\in\calC(S)^\Delta$,
\bqn
\Syst(\mu)>0\,.
\eqn
\end{Theorem}

\begin{proof}  First we show that if $\tau\in\calC(S)^\Delta$ with $\tau\neq0$,
then as current on $S$ it cannot be a measured lamination.  
Otherwise let $\calT(\tau)$ be the $\R$-tree dual to $\tau$.
Since $\tau$ is $\Delta$-invariant, $\Delta$ acts by isometries on the complete $\R$-tree $\calT(\tau)$.
The subset $\calT(\tau)^a$ and $\calT(\tau)^b$ of $a$-fixed, respectively $b$-fixed, points are not empty.
Since $\calT(\tau)$ is complete and $ab$ has a fixed point in $\calT(\tau)$,
this implies that $\calT(\tau)^a\cap\calT(\tau)^b\neq\varnothing$
and hence $\Delta$ has a fixed point in $\calT(\tau)$.  

On the other hand pick $\gamma\in\Gamma$ representing a closed geodesic $c$ in $S$ with $i(\tau,c)>0$.
Then in the tree $\calT(\tau)$ the element $\gamma$ acts as a translation of length $i(\tau,c)$ along its axis.
This is in contradiction with the fact that $\gamma$ has a fixed point.

Now let $\mu\in\calC(S)^\Delta$, $\mu\neq0$.  
Then the set $\calE_\mu$ of closed $\mu$-short solitary geodesics is $\Delta$-invariant,
and hence $\tau=\sum_{c\in\calE_\mu}\delta_c$ is a $\Delta$-invariant measured lamination on $S$,
which implies that $\calE_\mu=\varnothing$.  
Since $\mu\neq0$, Theorem~\ref{thm:1}(1) implies that $i(\mu,c)>0$ for every closed geodesic $c\subset S$.
If now $\Syst(\mu)=0$, Theorem~\ref{thm_intro:irred and syst 0}   implies that $\mu$ is a measured lamination,
which is impossible by the preceding discussion.  
Hence $\Syst(\mu)>0$, which concludes the proof.
\end{proof}

\begin{proof}[Proof of Corollary~\ref{cor_intro:1.10}]
  Let $([\rho_k])_k$ be a sequence in $\calX(\Gamma,G)$
  converging to a point $[L]$ of  $\Omega(\Gamma,G)$.
  Let $\mu_k\in\calC(S)$ be the current
with 
\bqn
i(\mu_k,c)=\|\lambda(\rho_k(\gamma))\|\,,
\eqn 
where $\gamma\in\Gamma$ is any hyperbolic element  
with corresponding closed geodesic $c\subset S$.
By \cite[Corollary~1.5]{Martone_Zhang}, we have that
\bqn
h(\rho_k)\Syst(\mu_k)\leq C\,,
\eqn
where $C$ is a constant only depending on $S$.
We will now show that
\bqn
\lim_{k\to\infty}\Syst(\mu_k)=\infty\,,
\eqn
which will imply the corollary.
%
%% New version (Anne 6/2/19)
Let  $b$ be the geodesic current associated to a binding multicurve in $S$.
Then, by compactness of
$\{\mu\in\calC(S):\ i(\mu, b)=1\}$,
% the subset of $\calC(\Sigma)$ consisting of the
% currents $\mu$ such that $i(\mu,b)=1$,
up to extracting we have that $\frac{1}{i(\mu_k,b)}\mu_k$ converges to some non
zero current $\mu$.  Then by continuity of $i$ we have
$$\lim_{k\to\infty}\frac{1}{i(\mu_k,b)}\|\lambda(\rho_k(\gamma))\|=i(\mu,c)$$
for all hyperbolic $\gamma\in\Gamma$   
with corresponding closed geodesic $c\subset S$.
As $\mu\neq 0$ this implies that  $\lim_{k\to\infty}i(\mu_k,b)=+\infty$
and $\|L(\gamma)\|=i(\mu,c)$ for all $\gamma$ for some representant
$L$ of $[L]$. In particular we have $\Syst(\mu)=\Syst(L)$. By continuity of the
systole  (Corollary~\ref{cor_intro:MCG}(1)) we have that
$$\lim_{k\to\infty} \frac{1}{i(\mu_k,b)}\Syst(\mu_k)=\Syst(\mu)$$
As $\Syst(\mu)=\Syst(L)>0$ and $\lim_{k\to\infty}i(\mu_k,b)=+\infty$,
this implies that $\lim_{k\to\infty}\Syst(\mu_k)=+\infty$.
%%% Old version (before 6 feb 2019) :
% Let $F\subset\Gamma$ be a finite subset satisfying the conclusions of \cite[Corollary~5.7]{Parreau12},
% that is
% \bqn
% \lim_{k\to\infty}[\rho_k]=x\in\partial\calX(\Gamma,G)
% \eqn
% if and only if, for some representative $L$ of $x$  and for all $\gamma\in\Gamma$,
% \bqn
% \lim_{k\to\infty}\frac{\|\nu(\rho_k(\gamma))\|}{\sum_{\eta\in F}\|\nu(\rho_k(\eta))\|}=L(\gamma)
% \eqn
% and
% \bqn
% \lim_{k\to\infty}\sum_{\eta\in F}\|\nu(\rho_k(\eta))\|=\infty\,.
% \eqn
% We may assume that $F$ is large enough so that if $\tau$ is the current, obtained as
% sum of Dirac masses on the closed geodesics corresponding to elements of $F$,
% we have $i(\tau,\mu)>0$ for every non-vanishing current $\mu$.  Thus
% \bqn
% \left\{\frac{\mu_k}{i(\mu_k,\tau)}\in\calC(S):\,k\geq1\right\}
% \eqn
% is relatively compact in $\calC(S)$
% and any accumulation point $\mu$ will satisfy $i(\mu,c)=L(\gamma)$ for all $\gamma\in\Gamma$.
% By Otal's theorem, \cite{Otal}, $\mu$ is uniquely determined by $L$ and hence
% \bqn
% \lim_{k\to\infty}\frac{\mu_k}{i(\mu_k,\tau)}=\mu\,.
% \eqn
% By Corollary~\ref{cor_intro:MCG}(1) we deduce that
% \bqn
% \lim_{k\to\infty}\frac{\Syst(\mu_k)}{i(\mu_k,\tau)}=\Syst(\mu)
% \eqn
% and since by assumption $[\rho_k]$ converges to a point in $\Omega(\Gamma,G)$,
% we have that $\Syst(\mu)>0$.  
% This, together with 
% \bqn
% \lim_{k\to\infty}i(\mu_k,\tau)=\infty
% \eqn
% implies that $\lim_{k\to\infty}\Syst(\mu_k)=\infty$.
\end{proof}
Turning to Corollary~\ref{cor_intro:triangle}, 
we define now the Jordan projection 
\bqn
\lambda:\SL(n,\K)\to\mathfrak{C}
\eqn 
for $\K=\R[[X^{-1}]]$.  Let 
\bqn
\overline\K:=\bigcup_{\stackrel{q\geq1}{q\in\N}}\C[[X^{-1/q}]]
\eqn
be an algebraic closure of $\K$.  The valuation $v$ defined on $\K$ extends uniquely to $\overline\K$ by 
\bqn
v(\lambda)=-\frac{\ell}{q}
\eqn 
if 
\bqn
\lambda=\sum_{\stackrel{k\leq\ell}{k\in\Z}} a_kX^{k/q}
\eqn
for $a_\ell\neq0$.
Given $g\in\SL(n,\K)$, we order the eigenvalues $a_1,\dots,a_n$ of $g$ so that
\bqn
\lambda(g):=(-v(a_1),\dotsc,-v(a_n))\in\mathfrak{C}\,.
\eqn
Let $\|\,\cdot\,\|_2$ be the Euclidean norm restricted to $\mathfrak{C}$.  We have:

\begin{Lemma}\label{lem:7.3}  Let $\ell_{\calB_n(\K)}(g)$ denote the translation length of $g\in \SL(n,\K)$
computed with respect to the CAT(0)-metric in $\calB_n(\K)$.
Then
\bqn
\ell_{\calB_n(\K)}(g)=\|\lambda(g)\|_2\,.
\eqn
\end{Lemma}

\begin{proof}  Let $\K\subset\LL\subset\overline\K$ be the splitting field in $\overline\K$
of the characteristic polynomial of $g$.
Since $\K$ is complete with discrete valuation and $\LL$ is a Galois extension of $\K$,
the building $\calB_n(\K)$ embeds $\SL(n,\K)$-equivariantly in the building $\calB_n(\LL)$ of $\SL(n,\LL)$
as a convex subset, \cite[2.6]{Tits79}.  Therefore it suffices to show that 
\bqn
\|\lambda(g)\|_2=\ell_{\calB_n(\LL)}(g)\,.
\eqn
Let $g=s\,u=u\,s$ be the Jordan decomposition of $g$ with $s$ diagonalizable in a basis $\calE$ of $\LL^n$
and $u$ unipotent upper triangular in this basis.
As $\ell_{\calB_n(\LL)}(u)=0$, we have that 
\bqn
\ell_{\calB_n(\LL)}(g)=\ell_{\calB_n(\LL)}(s)\,.
\eqn
The action by $s$ on the apartment associated to the basis $\calE$ in the model of $\calB_n(\LL)$ of good norms on $\LL^n$
(see \cite[3.2.2]{Parreau99}) is given by translation by $\lambda(s)=\lambda(g)$, which completes the proof of the lemma.
\end{proof}

\begin{proof}[Proof of Corollary~\ref{cor_intro:triangle}]
Let 
\bqn
\overline\K^r=\bigcup_{\ell\geq1}\R[[X^{-1/\ell}]]\,,
\eqn
endowed with the order
\bqn
\sum_{k\leq\ell} a_kX^{k/q}>0
\eqn
if $a_\ell>0$.
This order is compatible with the valuation $v$ and $\overline\K^r$ is real closed.
Observe that $\overline\K=\overline\K^r(\sqrt{-1})$: 
thus, if $\lambda\in\overline\K$, then $\lambda\overline\lambda\in\overline\K^r$ is positive
and we denote 
\bqn
\|\lambda\|=\sqrt{\lambda\overline\lambda}\,.
\eqn
Let $\overline{\R(X)}^r$ be the real closure of $\R(X)$ in $\overline\K^r$.
Given $g\in\SL(n,\R(X))$, its eigenvalues $\lambda_1,\dotsc,\lambda_n$ lie in $\overline{\R(X)}^r(\sqrt{-1})$
and we order them so that
\bqn
\lambda(g)=(-v(|a_1|),\dotsc,-v(|a_n|))\in\mathfrak{C}\,.
\eqn
Observe that any $\lambda\in \overline{\R(X)}^r$ can be represented by a Puiseux series that is convergent at $\infty$.
As a result, if $\lambda>0$
\bqn
-v(\lambda)=\lim_{t\to\infty}\frac{\ln\lambda(t)}{\ln t}\,.
\eqn
This implies that for every $\gamma\in\Gamma$
\bqn
\lim_{t\to\infty}\frac{\lambda(\rho_t(\gamma))}{\ln t}=\lambda(\rho(\gamma))\,.
\eqn
Since $\tr(\rho(\gamma_0))$ has a pole at infinity,
we must have $\lambda(\rho(\gamma_0))\neq0$, 
from which it follows that $[\lambda\circ\rho]\in\partial\calX(\Delta,\SL(n,\R))$.
Since the latter coincides with $\Omega(\Delta,\SL(n,\R))$ by Theorem~\ref{thm_intro:triangle},
this shows (1).

For the second assertion, let $\Gamma<\Delta$ be a torsion-free finite index subgroup.
By Lemma~\ref{lem:7.3} we have that for all $\gamma\in\Gamma$ 
\bqn
\|\lambda(\rho(\gamma))\|_2=\ell_{\calB_n(\K)}(\rho(\gamma))\,,
\eqn
and Theorem~\ref{thm_intro:triangle} and Corollary~\ref{cor_intro:pd}(2) imply then
that for some constants $0<c_1\leq c_2$,
\bqn
c_1\ell_{\mathrm{hyp}}(\gamma)\leq\ell_{\calB_n(\K)}(\rho(\gamma))\leq c_2\ell_{\mathrm{hyp}}(\gamma)
\eqn
for all $\gamma\in\Gamma$.

This says that the $\Gamma$-action on $\calB_n(\K)$ is displacing and hence 
\cite[Proposition~2.2.2 and Lemma~2.8.1]{DGLM} imply that the $\Gamma$-orbits,
and hence the $\Delta$-orbits, are quasi-isometric embeddings.
\end{proof}


\begin{thebibliography}{DGLM11}
	
	\bibitem[ALS18]{ALS18}
	D.~{Alessandrini}, G.-S. {Lee}, and F.~{Schaffhauser}.
	\newblock {Hitchin components for orbifolds}.
	\newblock {\em ArXiv e-prints}, November 2018.
	
	\bibitem[BCLS18]{BCLS}
	M.~Bridgeman, R.~Canary, F.~Labourie, and A.~Sambarino.
	\newblock Simple root flows for {H}itchin representations.
	\newblock {\em Geom. Dedicata}, 192:57--86, 2018.
	
	\bibitem[BIPP20]{BIPP-ann}
	Marc Burger, Alessandra Iozzi, Anne Parreau, and Marie~Beatrice Pozzetti.
	\newblock The real spectrum compactification of character varieties:
	characterizations and applications, 2020.
	\newblock To appear \emph{C. R. Math. Acad. Sci. Paris }.
	
	\bibitem[BIPP21]{BIPP2}
	Marc Burger, Alessandra Iozzi, Anne Parreau, and Maria~Beatrice Pozzetti.
	\newblock Positive crossratios, barycenters, trees and applications to maximal
	representations, 2021.
	\newblock {https://arxiv.org/abs/2103.17161}.
	
	\bibitem[BIW10]{BIW}
	M.~Burger, A.~Iozzi, and A.~Wienhard.
	\newblock Surface group representations with maximal {T}oledo invariant.
	\newblock {\em Ann. of Math. (2)}, 172(1):517--566, 2010.
	
	\bibitem[BIW14]{BIW_htt}
	M.~Burger, A.~Iozzi, and A.~Wienhard.
	\newblock Higher {T}eichm\"{u}ller spaces: from {${\rm SL}(2,\Bbb R)$} to other
	{L}ie groups.
	\newblock In {\em Handbook of {T}eichm\"{u}ller theory. {V}ol. {IV}}, volume~19
	of {\em IRMA Lect. Math. Theor. Phys.}, pages 539--618. Eur. Math. Soc.,
	Z\"{u}rich, 2014.
	
	\bibitem[Bon86]{Bon86}
	F.~Bonahon.
	\newblock Bouts des vari\'et\'es hyperboliques de dimension {$3$}.
	\newblock {\em Ann. of Math. (2)}, 124(1):71--158, 1986.
	
	\bibitem[Bon88]{Bon88-curr}
	F.~Bonahon.
	\newblock The geometry of {T}eichm\"uller space via geodesic currents.
	\newblock {\em Invent. Math.}, 92(1):139--162, 1988.
	
	\bibitem[Bon01]{Bon88-lam}
	F.~Bonahon.
	\newblock Geodesic laminations on surfaces.
	\newblock In {\em Laminations and foliations in dynamics, geometry and topology
		({S}tony {B}rook, {NY}, 1998)}, volume 269 of {\em Contemp. Math.}, pages
	1--37. Amer. Math. Soc., Providence, RI, 2001.
	
	\bibitem[BP17]{BP}
	M.~Burger and M.B. Pozzetti.
	\newblock Maximal representations, non-{A}rchimedean {S}iegel spaces, and
	buildings.
	\newblock {\em Geom. Topol.}, 21(6):3539--3599, 2017.
	
	\bibitem[CEG06]{notesonnotes}
	R.~D. Canary, D.~B.~A. Epstein, and P.~L. Green.
	\newblock Notes on notes of {T}hurston [mr0903850].
	\newblock In {\em Fundamentals of hyperbolic geometry: selected expositions},
	volume 328 of {\em London Math. Soc. Lecture Note Ser.}, pages 1--115.
	Cambridge Univ. Press, Cambridge, 2006.
	\newblock With a new foreword by Canary.
	
	\bibitem[DGLM11]{DGLM}
	Th. Delzant, O.~Guichard, F.~Labourie, and S.~Mozes.
	\newblock Displacing representations and orbit maps.
	\newblock In {\em Geometry, rigidity, and group actions}, Chicago Lectures in
	Math., pages 494--514. Univ. Chicago Press, Chicago, IL, 2011.
	
	\bibitem[DLR10]{DLR}
	M.~Duchin, Ch.J. Leininger, and K.~Rafi.
	\newblock Length spectra and degeneration of flat metrics.
	\newblock {\em Invent. Math.}, 182(2):231--277, 2010.
	
	\bibitem[Ebe96]{Eb96}
	P.B. Eberlein.
	\newblock {\em Geometry of nonpositively curved manifolds}.
	\newblock Chicago Lectures in Mathematics. University of Chicago Press,
	Chicago, IL, 1996.
	
	\bibitem[EM]{EM18}
	V.~Erlandsson and G.~Mondello.
	\newblock Ergodic invariant measures on the space of geodesic currents.
	\newblock https://arxiv.org/abs/1807.02144.
	
	\bibitem[Glo]{Glorieux}
	O.~Glorieux.
	\newblock Critical exponent for geodesic currents.
	\newblock https://arxiv.org/abs/1704.06541.
	
	\bibitem[GM91]{GaMa91}
	F.~P. {Gardiner} and H.~{Masur}.
	\newblock {Extremal length geometry of Teichm\"uller space.}
	\newblock {\em {Complex Variables, Theory Appl.}}, 16(2-3):209--237, 1991.
	
	\bibitem[Gol88]{Goldman_boulder}
	W.M. Goldman.
	\newblock Geometric structures on manifolds and varieties of representations.
	\newblock In {\em Geometry of group representations ({B}oulder, {CO}, 1987)},
	volume~74 of {\em Contemp. Math.}, pages 169--198. Amer. Math. Soc.,
	Providence, RI, 1988.
	
	\bibitem[Lab06]{Lab06}
	F.~Labourie.
	\newblock Anosov flows, surface groups and curves in projective space.
	\newblock {\em Invent. Math.}, 165(1):51--114, 2006.
	
	\bibitem[LRT11]{LRT}
	D.~Long, A.~Reid, and M.~Thistlethwaite.
	\newblock Zariski dense surface subgroups in {${\rm SL}(3,\bold Z)$}.
	\newblock {\em Geom. Topol.}, 15(1):1--9, 2011.
	
	\bibitem[Mar]{Martelli}
	B.~Martelli.
	\newblock An introduction to geometric topology.
	\newblock https://arxiv.org/abs/1610.02592.
	
	\bibitem[MS91]{MS91}
	J.W. Morgan and P.B. Shalen.
	\newblock Free actions of surface groups on {${\bf R}$}-trees.
	\newblock {\em Topology}, 30(2):143--154, 1991.
	
	\bibitem[MZ19]{Martone_Zhang}
	Giuseppe Martone and Tengren Zhang.
	\newblock Positively ratioed representations.
	\newblock {\em Comment. Math. Helv.}, 94(2):273--345, 2019.
	
	\bibitem[Nie15]{Nie}
	X.~Nie.
	\newblock Entropy degeneration of convex projective surfaces.
	\newblock {\em Conform. Geom. Dyn.}, 19:318--322, 2015.
	
	\bibitem[Ota90]{Otal}
	J.-P. Otal.
	\newblock Le spectre marqu\'e des longueurs des surfaces \`a courbure
	n\'egative.
	\newblock {\em Ann. of Math. (2)}, 131(1):151--162, 1990.
	
	\bibitem[Par00]{Parreau99}
	A.~Parreau.
	\newblock Immeubles affines: construction par les normes et \'{e}tude des
	isom\'{e}tries.
	\newblock In {\em Crystallographic groups and their generalizations
		({K}ortrijk, 1999)}, volume 262 of {\em Contemp. Math.}, pages 263--302.
	Amer. Math. Soc., Providence, RI, 2000.
	
	\bibitem[Par12]{Parreau12}
	A.~Parreau.
	\newblock Compactification d'espaces de repr\'esentations de groupes de type
	fini.
	\newblock {\em Math. Z.}, 272(1-2):51--86, 2012.
	
	\bibitem[Tit79]{Tits79}
	J.~Tits.
	\newblock Reductive groups over local fields.
	\newblock In {\em Automorphic forms, representations and {$L$}-functions
		({P}roc. {S}ympos. {P}ure {M}ath., {O}regon {S}tate {U}niv., {C}orvallis,
		{O}re., 1977), {P}art 1}, Proc. Sympos. Pure Math., XXXIII, pages 29--69.
	Amer. Math. Soc., Providence, R.I., 1979.
	
	\bibitem[Zha15a]{Tengren1}
	T.~Zhang.
	\newblock The degeneration of convex {$\Bbb{R}\Bbb{P}^2$} structures on
	surfaces.
	\newblock {\em Proc. Lond. Math. Soc. (3)}, 111(5):967--1012, 2015.
	
	\bibitem[Zha15b]{Tengren2}
	T.~Zhang.
	\newblock Degeneration of {H}itchin representations along internal sequences.
	\newblock {\em Geom. Funct. Anal.}, 25(5):1588--1645, 2015.
	
\end{thebibliography}
\end{document}